\documentclass[11pt]{amsart}
\usepackage{graphicx,amssymb,amsfonts,amsmath,amsthm,newlfont}
\usepackage{epsfig,url}
\usepackage{color}
\usepackage{blkarray}
\usepackage{enumitem}
\usepackage{fancyhdr}

\usepackage[all,2cell]{xy} \UseAllTwocells \SilentMatrices

\newtheorem{thm}{Theorem}[section]

\newtheorem{lem}[thm]{Lemma}
\newtheorem{prop}[thm]{Proposition}
\theoremstyle{definition}
\newtheorem{defn}[thm]{Definition}
\newtheorem{conj}{Conjecture}

\newtheorem{ex}[thm]{Examples}
\newtheorem{example}[thm]{Example}

\theoremstyle{remark}
\newtheorem{rem}[thm]{Remark}

\numberwithin{equation}{section}
\newcommand{\Z}{\mathbb Z}
\newcommand{\C}{\mathbb C}

\newcommand{\R}{\mathbb R}
\newcommand{\N}{\mathbb N}
\newcommand{\Pro}{\mathbb P}

\newcommand{\gr}{\mathrm{gr}}

\font \rus= wncyr10
\newcommand{\sha}{\, \hbox{\rus x} \,}

\newcommand{\MT}{\mathcal{MT}}

\newcommand{\rel}{\mathrm{rel}}

\newcommand{\ML}{\Lambda}

\newcommand{\Q}{\mathbb Q}

\newcommand{\To}{\longrightarrow}

\newcommand{\G}{\mathbb{G}}

\newcommand{\tone}{\overset{\rightarrow}{1}\!}

\newcommand{\SL}{\mathrm{SL}}

\newcommand{\Or}{\mathcal{O}}

\newcommand{\Mel}{\mathfrak{M}}

\newcommand{\mm}{\mathfrak{m} }

\newcommand{\EE}{\mathbf{E}}

\begin{document}
\author{Francis Brown}

\begin{title}[Multi-variable  $L$-functions]{A multi-variable version of   the completed Riemann zeta function and other   $L$-functions}\end{title}

\maketitle
\section{introduction}

This is the continuation of my talk at Professor Ihara's birthday conference, but for the most part, is   logically independent from it. It is  an attempt to find a general definition of multiple $L$-functions. The hope  is to obtain certain periods of algebraic varieties by combining the data of  the traces of Frobenius (or point counts over finite fields)  into an  analytic function of several complex variables. With this distant goal in mind, we define a  tentative class of multiple motivic $L$-functions which are meromorphic functions of several variables  satisfying a functional equation and multiplicative shuffle identities. 
In the simplest case of the trivial motive $\Q=H^0 (\mathrm{Spec}\, \Q)$, this yields a  multi-variable version $\xi(s_1,\ldots, s_r)$ of the Riemann $\xi$-function:
$$\xi(s) = \pi^{-s/2} \Gamma(s/2) \zeta(s)\ .$$
The general theory,   in  this particular case, yields the following theorem.
  \begin{thm} \label{thmRiemannxi} The function $\xi(s_1,\ldots, s_r)$ is  meromorphic  on  $\C^r$,  and satisfies  a functional equation:
  $$\xi(s_1,\ldots, s_r) = \xi(1-s_r,\ldots, 1-s_1) $$
and  shuffle  product identities:
$$
\xi( s_1, \ldots, s_p) \xi(  s_{p+1}, \ldots, s_{p+q})   =  \sum_{\sigma \in \Sigma_{p,q}}  \xi( s_{\sigma(1)}, \ldots, s_{\sigma(p+q)})\ \ .$$
It has simple poles along the hyperplanes 
$$ s_1+ \ldots +s_k = k \qquad  ,    \qquad s_k+\ldots+s_r=0  \quad \hbox{ for all } \quad 1\leq k \leq r \ , $$ 
and its residues have  the  recursive structure:
$$
\mathrm{Res}_{s_k+ \ldots + s_r=0} \, \xi(s_1,\ldots, s_r) = (-1)^{r-k+1}\, \frac{    \xi(s_1,\ldots, s_{k-1})}{   (s_{k+1} + \ldots + s_r)  \ldots (s_{r-1}+s_{r})s_r   }\  .
$$
\end{thm}
The functions $\xi(s_1,\ldots, s_r)$ are not obviously related  to multiple zeta functions, for which no functional equation is  presently known to exist. Furthermore,  the double $\xi$-values $\xi(2\ell_1, 2\ell_2)$  for even integers $2\ell_1,2\ell_2>0$ are related to periods of simple extensions of  symmetric powers of the cohomology of the elliptic curve  $\C/ \Z \oplus \Z i$ which has complex multiplication. 

 We provide  some evidence in support of the philosophy outlined above by   proving that all periods of mixed Tate motives over $\Z$ can be expressed as totally critical values of some simple multiple motivic $L$-functions, and furthermore  by showing  that totally holomorphic multiple modular values (periods of the relative completions of modular groups) can  also be subsumed into this framework.  It therefore seems,  at least in some simple cases, that  values of multiple $L$-functions can indeed predict the periods of mixed motives which are beyond the reach of existing conjectures on special values of  ordinary   $L$-functions.

\section{$L$-functions and Mellin transforms} \label{sectLandMellin}
We briefly and informally recall the main properties of motivic $L$-functions and their associated theta functions \cite{Dokchister}. See 
 \cite{DeL, Nekovar} for further details.   We shall use the word `motive' loosely, as is customary in this area, since the $L$-function of a motive depends only upon its realisations, and its properties are largely conjectural.  In any case, most of our examples concern situations where the $L$-function is completely classical. 

\subsection{Motivic $L$-functions} 
 To a pure motive $M$ over $\Q$ of weight $m\geq 0$,  one attaches
\begin{itemize}
\item a Dirichlet series, defined as an Euler product
$$L(M;s) = \prod_{p \,\,\mathrm{prime}} L_p(M;s)  = \sum_{n\geq 1} \frac{a_n}{n^s}\ ,$$
which is assumed to converge for $\mathrm{Re}(s)$ sufficiently large.
\vspace{0.1in}
\item a completed $L$-function, defined by \cite{SerreGamma} 
$$L^*(M;s) = L_{\infty}(M;s) L(M;s)\ , $$
where $L_{\infty}(M;s)$ is a finite  product of factors $\Gamma_{\R/\C}(s-n)$ where
 $n$ is an integer, and 
$\Gamma_{\R}(s) = \pi^{-s/2} \Gamma(s/2)$, $\Gamma_{\C}(s) = (2\pi)^{-s} \Gamma(s)$.
\vspace{0.1in}
\item  One hopes that  $L^*(M;s)$ admits a meromorphic continuation to the complex plane and satisfies a functional equation of the form 
$$L^*(M;s )  =    \epsilon(M;s)\, L^*(M^{\vee}(1); -s)\  , $$
for $\epsilon(M;s)$  of the form $a   N^s$, where $N>0$ is an integer.
\end{itemize} 
From now on, we shall   restrict to the case when $M$ is self-dual, i.e., $M^{\vee} = M(m)$, for then  the expected functional equation reduces to:
$$L^*(M;s) = \epsilon(M;s) L^*(M; m+1-s)\ .  $$
This restriction is by no means necessary, but simplifies the exposition.

\begin{example}
Let  $M= H^m(X;\Q)$ where $X$ is a  smooth projective variety over $\Q$. Then
$$L_p(M;s) = \det \left( 1- F_p p^{-s} \big| M_{\ell}^{I_p}\right)^{-1}$$
for some prime $\ell \neq p$,  where $M_{\ell}= H^m_{et}(X \otimes_{\Q} \overline{\Q}; \Q_{\ell})$, $F_p$ is the geometric Frobenius, and $I_p \leq \mathrm{Gal}(\overline{\Q}/\Q)$  the inertia subgroup.   It is assumed to be independent of   $\ell$ at primes of bad reduction. By Deligne, $L(M;s)$  converges for   $\mathrm{Re}(s)> 1+m/2$. Serre \cite{SerreGamma} (15) defines $L_{\infty}(M;s)$ in terms of the real Hodge structure of $M$. Note that the motive  $M$ is self-dual.
\end{example} 
One can  consider motives over more general   number fields, but 
by restriction of scalars, one can always reduce to the field of rationals $\Q$. 

Recall the duplication formula:
\begin{equation} \label{Duplication}
2 \, \Gamma_{\C}(s) = \Gamma_{\R}(s) \Gamma_{\R}(s+1) \ .
\end{equation} 
The values of $\Gamma_{\R/\C}(s)$ at positive integers are integer powers of $\pi$ times a   rational number.

\subsection{Reformulation} Let us fix a self-dual motive $M$. It is convenient  to simplify the functional equation  by   rescaling  the completed $L$-function  as follows:
$$\Lambda(M;s)  = (\sqrt{N})^{-s}  L^*(M;s)$$  where $\sqrt{N}$ is the positive root of $N$. Set $\varepsilon_M = a (\sqrt{N})^{m+1}$. 
 Following  the notations of  \cite{Dokchister}, one can write this function in the form
$$\Lambda(M;s)  = A^s  \gamma(s) L(M;s) $$
where  the product of   gamma factors occuring in  $L_{\infty}(M;s)$ is denoted by
$\gamma(s) = \prod_{i=1}^d  \Gamma \left( \frac{s+ \lambda_i}{2} \right)$
for some integers $\lambda_i \in \Z$ which only depend on the Hodge numbers of $M$ and the action of the real Frobenius (complex conjugation) on its Betti realisation. The integer $d$ is equal to the rank of $M$.
The number  $A$ is positive and  real, and  absorbs  both the powers of $\pi$ and the exponential   factors occuring in $\epsilon(M;s)$.    
This partitioning into $A^s$ and $\gamma(s)$ is arbitrary; for example,  one could  demand that $\gamma$ be a product of functions $\Gamma_{\R}(s)$ without substantively affecting the following discussion. 

  In order to define the  theta function   \cite{Dokchister}  of $M$  one only needs the fact that $L(M;s)$  converges for $\mathrm{Re}(s)$ sufficiently large, together with the following assumptions:   \begin{itemize}
  \item That  $\Lambda(M;s)$ 
   admits a meromorphic continuation to $\C$  and is bounded in vertical strips. It has  finitely many poles, all of which are simple. 
  \vspace{0.1in}
  \item That, for some $\varepsilon_M \in \C$,  necessarily $\pm 1$,   it admits   a functional equation 
   \begin{equation} \label{LambdaMfunctional} \Lambda(M;s) = \varepsilon_M \,\Lambda(M; m+1-s) \ .\end{equation}
   This  equation is  equivalent to the functional equation for $L^*(M;s)$. 

\end{itemize} 
The Euler product will play very little  role. We shall also assume that the poles of  $\Lambda(M;s)$ are integers. This requirement  is not essential and can easily be  relaxed. 

\subsection{Theta functions} \label{sectThetadef}
 In \cite{Dokchister} Dokchister defines a continuous function $\phi(t)$ on the positive real axis to be the inverse Mellin transform of $\gamma(s)$, i.e., 
$$\gamma(s) = \int_0^{\infty} \phi(t) t^{s} \frac{dt}{t}  \  .$$
It depends only on the  $\lambda_i \in \Z$ and  tends to zero exponentially fast as $t\rightarrow \infty$. It can be expressed in terms of hypergeometric functions and can be computed once and for all for any given class of motives. 
 He then defines the associated theta function 
 $$\theta_M(t) = \theta_M^{\infty} (t) + \theta^0_M(t)$$
 where $\theta^{\infty}_M(t) \in \C[t]$ is a polynomial in $t$  and 
 \begin{equation} \theta^0_M(t) = \sum_{n\geq 1}   a_n \phi\left( \frac{nt}{A}\right) 
 \end{equation}
is a generalised theta function which converges exponentially fast as $t\rightarrow \infty$. The  $a_n$ are the coefficients in the Dirichlet series $L(M;s)$ and were assumed to have polynomial growth.  The completed $L$-function is  then the Mellin transform of $\theta^0_M(t)$:
$$\Lambda(M;s) = \int_0^{\infty}  \theta^0_M(t)   t^{s} \frac{dt}{t} \ .$$
The  polynomial  $\theta_M^{\infty}(t)$ is determined from the poles of $\Lambda(M;s)$ and its residues. In particular, it vanishes whenever $\Lambda(M;s)$ has no poles. That it is a polynomial is equivalent to the fact that $\Lambda(M;s)$ has poles only at integer points:
 were $\Lambda(M;s)$ to have  poles in $\Z[1/n]$ for some $n$, the function $\theta_M^{\infty}$ would need to be replaced with  a polynomial in $t^{1/n}$.  
 
 The functional equation of $\Lambda(M;s)$ is then equivalent to the inversion formula
\begin{equation} \theta_M(t^{-1} )  = \varepsilon_M\, t^{m+1} \theta_M(t)\ .
\end{equation}
Since $m\geq 0$,  one checks that this equation uniquely determines  $\theta_M^{\infty}(t)$ from $\theta^0_M(t)$.

\begin{example} (Riemann zeta function). Let $X=\mathrm{Spec} \, \Q$ be a point. Then  $M =H^0(X)=\Q$ is  the trivial motive. Its $L$-function is the Riemann zeta function
$$\zeta(s) = \prod_{p\,  \mathrm{prime}}  (1- p^{-s})^{-1} = \sum_{n\geq 1 }  \frac{1}{n^s}\ ,$$
 for all $\mathrm{Re} (s)>1$. Its completed version $\xi(s)  = \pi^{-s/2} \Gamma(s/2) \zeta(s)$
admits a meromorphic continuation to $\C$ with  simple poles at $s=0,1$ and satisfies 
$\xi(s) = \xi(1-s)$. Let 
\begin{equation} \theta_{\Q}(t) =   \sum_{n\in \Z} e^{- \pi n^2 t^2} 
\end{equation}
denote its inverse Mellin transform. It is  essentially the restriction of the Jacobi theta function to the imaginary axis.  It can be written as a sum $ \theta_{\Q}=
 \theta^{\infty}_{\Q}+ \theta^0_{\Q}$, where
$$\theta_{\Q}^{\infty}(t) =1 \qquad \hbox{ and } \qquad  \theta_{\Q}^{0}(t) =2 \sum_{n\geq 1} e^{- \pi n^2 t^2} \ .$$
Since
$$ \pi^{-s/2} \Gamma(s/2)  = 2 \int_0^{\infty}   e^{- \pi t^2}  t^{s} \frac{dt}{t}\ ,$$
we deduce  that for all $\mathrm{Re}(s)>1$: 
$$\xi(s) =  \int_0^{\infty}  \left(\theta_{\Q}(t) -1 \right) t^{s} \frac{dt}{t} \ .$$
The functional equation of $\xi$ is equivalent to 
$\theta_{\Q}(t^{-1})  = t \theta_{\Q}(t)$. 
The formula for $\xi(s)$ looks strange at first sight: 
one integrates the truncated function $\theta^{\infty}_{\Q}(t) = \theta_{\Q}(t)-1$, although it   is  $\theta_{\Q}$ which satisfies the   inversion formula. Using   tangential base points, we 
 shall interpret $\xi$ as a  regularised Mellin transform of  the full function $\theta_{\Q}$, which will make the functional equation obvious. 
\end{example}

\begin{example} (Cusp forms of level 1). Let 
$f(\tau)=\sum_{n\geq 1} a_n e^{2\pi i n \tau}$ be a cusp form of weight $2k$ for the full modular group $\mathrm{SL}_2(\Z)$, and an eigenfunction for Hecke operators with $a_1=1$. 
 Scholl \cite{Scholl} has shown how to associate a pure motive $M_f$ to $f$, which has coefficients in the field $K_f$ generated by the $a_n$.    It has weight $2k-1$ and is of Hodge type $(2k-1,0)$ and $(0,2k-1)$. 
 The associated $L$-function is
$$L(f; s) =  \prod_{p}  \left(1- a_p p^{-s} +  p^{2k-1-2s}\right)^{-1}  =  \sum_{n\geq 1} \frac{a_n}{n^s}   \ , $$
and converges for $\mathrm{Re}(s)> k+1$. 
The completed $L$-function,  defined by Hecke, is 
$$\Lambda(f; s) = (2\pi)^{-s} \Gamma(s)  L(f; s)\ .$$
It extends to an entire function on $\C$ satisfying  
$\Lambda(f;s) = (-1)^k \Lambda(f;2k-s)$.
Its inverse Mellin transform is the restriction of $f$ to the positive imaginary axis:
$$\theta_{f}(t) = \theta^0_{f}(t) = \sum_{n \geq 1} a_n e^{-2 \pi t}\ . $$
Here we have  $\theta^{\infty}_{f}(t)=0$. The inversion formula is  $\theta_{f}(t^{-1}) = (-1)^k t^{2k} \theta_{f}(t)$. One has 
 $$\Lambda(f;s) = \int_0^{\infty} \theta_{f}(t)  t^{s} \frac{dt}{t} \ . $$
\end{example}
 
\subsection{Conjectures on special values of $L$-functions}  \label{sectConjL} This  is a vast subject  originating from   Euler's 
formula for $\zeta(2n)$, and is based on a huge  range of  examples which have been gathered over the intervening two and a half centuries. We shall be extremely brief  and deliberately vague, and refer instead to \cite{Nekovar} for a recent survey.

A key definition \cite{DeL} is that of a critical point. For $M$ as above, Deligne defines an integer $n$ to be critical if neither $L_{\infty}(M;s)$  
nor $L_{\infty}(M^{\vee}(1);s)$ has a pole at $s=n$.

\begin{itemize}
\item For \emph{critical} $n$, Deligne's conjecture predicts that $L(M;n)$ should be related to a period of the  motive $M$ (or rather, Tate twists of its dual).
\item  For \emph{non-critical} $n$, Beilinson's conjecture \cite{Beilinson} predicts in most cases that $L(M;n)$ should be related to periods not of (Tate twists of) $M$.
 \end{itemize} 
For certain exceptional values of $s$, Beilinson's conjecture relates $L(M;s)$ to a  biextension of $\Q,\Q(1)$ and $M(n)$, but  this case will not play any further role in this write-up. 

In summary, these conjectures and their generalisations provide an interpretation for 
the special  values of $L(M;s)$ at  integers as periods of mixed motives of a very simple kind. 
Viewed upside down, these conjectures give  a  \emph{formula} for certain periods of pure motives (in the case of Deligne's conjecture), and
for certain periods of simple extensions (in the case of Beilinson's conjecture) in terms of $L$-values.

\subsection{Speculation} 

It is tempting to wonder if  this might be part of a larger picture relating periods of more general mixed motives and values of `mixed $L$-functions':
$$ \{\hbox{Periods of mixed motives} \}   \quad \overset{?}{\longleftrightarrow} \quad \{ \hbox{Mixed L-values}\}$$
Such a formalism would   interpret certain periods of an iterated extension  of pure motives as values  of analytic functions constructed out of the action  of  the Frobenius operators on the $\ell$-adic realisations of its constituent  motives.

\section{Iterated  Mellin transforms}  \label{sectIteratedMellin}
Consider a set of  functions $\theta_1,\ldots, \theta_r$ which are continuous on the positive real axis, and have the following properties:

\begin{enumerate}
\item  The existence of a functional equation for all $i$:
$$\theta_i (t^{-1}) = \varepsilon_i t^{w_i} \theta_{i}(t)\ \ .$$
\item The existence of a decomposition of the form:
$$\theta_i = \theta_i^{\infty} + \theta_i^0   \  , $$
for all $i$, where $\theta_i^{\infty} \in \C[t]$ and $\theta_i^{0}$ tends to  zero exponentially fast  as $t\rightarrow \infty$. 
\end{enumerate} 
These conditions are satisfied for the inverse Mellin transform of a motivic $L$-function which satisfies the assumptions detailed in the previous paragraph, and   has at most simple poles at integers.  We shall enlarge the class of $\theta$ functions that we wish to consider in \S \ref{sectNewTheta}. 

 We presently explain how to define a multiple  $\ML$-function:
$$\ML(\theta_1,\ldots, \theta_r; s_1,\ldots, s_r)$$
and prove its basic properties, which are summarised in the following theorem.

\begin{thm}  \label{thmLambdaMainProperties} The functions $\Lambda(\theta_1,\ldots, \theta_r; s_1,\ldots, s_r)$ are meromorphic on $\C^r$ with at most simple poles along  hyperplanes which depend only  on the polynomials  $\theta_i^{\infty}$.  They have no poles when  the  $\theta_i^{\infty}$ vanish for all $1\leq i \leq r$.
They satisfy a functional equation:
$$\ML(\theta_1,\ldots, \theta_r; s_1,\ldots, s_r) =   \varepsilon_1\ldots \varepsilon_r \,  \ML(\theta_r,\ldots, \theta_1; w_r-s_r, \ldots, w_1-s_1)$$
(note the reversed order of the arguments) 
and shuffle product formula
\begin{multline} \ML(\theta_1,\ldots, \theta_p;  s_1, \ldots, s_p) \ML(\theta_{p+1},\ldots, \theta_{p+q};  s_{p+1}, \ldots, s_{p+q})  \nonumber \\ 
=  \sum_{\sigma \in \Sigma_{p,q}}  \ML(\theta_{\sigma(1)}, \ldots, \theta_{\sigma(p+q)}, s_{\sigma(1)}, \ldots, s_{\sigma(p+q)})
\end{multline}
where $\Sigma_{p,q}$ denotes the set of $(p,q)$-shuffles.
\end{thm}

One can rescale this multi-variable $L$-function by  defining 
\begin{equation}  \label{LstarversusLambda}
L^*(\theta_1,\ldots, \theta_r; s_1,\ldots, s_r) = N_1^{s_1/2} \ldots N_r^{s_r/2} \, \Lambda(\theta_1,\ldots, \theta_r; s_1,\ldots s_r) \ ,
\end{equation} 
where $N_i$ is the exponential factor in $\epsilon(M_i;s)$ and  $M_i$ is the motive corresponding to  $\theta_i$.  
This  function has similar properties to $\ML$, but with  a slightly more complicated functional equation  (replace every $\varepsilon_i$  in the above with $\epsilon(M_i;s)$). For $r=1$ it reduces to the  definition of  \S\ref{sectLandMellin}.
 
We first consider the simpler  situation where all  $\theta_i^{\infty}$ vanish, in which case the previous theorem is an immediate consequence of the theory of iterated integrals. The main issue in the general case is  to regularise divergences  correctly. For this  we use  a modification (\cite{MMV}, \S4) of Deligne's theory   of tangential base points (\cite{DePi1}, \S15).

\subsection{Multiple Mellin transforms in the case with  no poles}   For $s_i \in \C$,  let us write
$$\underline{\theta}^{\bullet}_i (s_i)= \theta^{\bullet}_i(t) t^{s_i-1} dt  \quad \ , \quad \quad  \hbox{ for } \quad  \bullet  =  \emptyset, 0, \infty  $$
to denote the one-forms associated to the $\theta_i$.   For levity of notation, we shall sometimes simply write $\underline{\theta}^{\bullet}_i $ for $\underline{\theta}^{\bullet}_i (s_i)$. We first assume  all $\theta^{\infty}_i$  vanish. 

\begin{defn} Suppose that $\underline{\theta}_i^{\infty}=0$ for all $i = 1,\ldots, r$.
 Define
 \begin{eqnarray}  \nonumber 
\ML(\theta_1,\ldots, \theta_r; s_1,\ldots, s_r) & = & \int_0^{\infty}   \underline{\theta}_1(s_1)  \, \cdots   \, \underline{\theta}_{r}(s_r) \qquad \qquad \qquad \hbox{(Iterated integral)}   \\
&= &  \int_{0\leq t_1 \leq t_2 \leq \ldots \leq t_r \leq \infty}   \theta_{1}(t_1) t_1^{s_1-1}dt_1 \ldots \theta_{r}(t_r)  t_r^{s_r-1}dt_r \end{eqnarray} 
The integral converges for all  $s_i \in \C$. This follows from the   exponential decay  of the functions $\theta_i$ at infinity and at $0$, which follows  from the inversion formula.
   \end{defn}

\begin{prop} \label{prop: propertiesnopoles} The functions $\ML$   satisfy the formula:
\begin{multline} \label{LambdaasR}    \ML(\theta_1,\ldots, \theta_r; s_1,\ldots, s_r) = \\
 \sum_{k=0}^r   \varepsilon_1\ldots \varepsilon_k R(\theta_k,\ldots, \theta_1; w_k-s_k,\ldots,w_1- s_1)  R(\theta_{k+1}, \ldots, \theta_r; s_{k+1}, \ldots,  s_r) 
 \end{multline} 
where the functions $R$  are defined  by the iterated integrals 
\begin{eqnarray}  
R(\theta_1,\ldots, \theta_r; s_1,\ldots, s_r) &  = & \int_1^{\infty} \underline{\theta}_1 (s_1)\ldots \underline{\theta}_r (s_r)  \nonumber \\
 &  = &  \int_{1 \leq t_1\leq  \ldots \leq t_r\leq \infty}  \theta_{1}(t_1) t_1^{s_1-1}dt_1 \ldots \theta_{r}(t_r)  t_r^{s_r-1}dt_r \nonumber
\end{eqnarray} 
which converge, and are holomorphic, for all $s_1,\ldots, s_r \in \C$. If $r=0$ then $R$ is defined to be $1$.
In particular,  $\Lambda(\theta_1,\ldots, \theta_r;s_1,\ldots,s_r) $   is analytic on $\C^r$. The shuffle product formula and functional equations  stated in theorem \ref{thmLambdaMainProperties} hold.

\end{prop}
\begin{proof} For sufficiently large $\mathrm{Re}(s_i)$, apply the composition of paths   formula \cite{Ch}
$$ \int_{0}^{\infty} \underline{\theta}_1\ldots \underline{\theta}_r =  \sum_{k=0}^r \int_{0}^1 \underline{\theta}_1 \ldots \underline{\theta}_k \int_{1}^{\infty} \underline{\theta}_{k+1}\ldots \underline{\theta}_r$$
to the definition of $\Lambda$. 
Apply the change of variables $t\mapsto t^{-1}$ to the left-hand integrals from $0$ to $1$ on the right  of the equality sign  and invoke  the inversion formula for the  $\theta_{i}$.  
Now write   $\underline{\tilde{\theta}}_i=\theta_i t^{w_i-s_i-1} dt$ and apply the reversal of paths formula \cite{Ch}  
$$    (-1)^k \int_{\infty}^1 \underline{\tilde{\theta}}_1 \ldots \underline{\tilde{\theta}}_k  =    \int_1^{\infty} \underline{\tilde{\theta}}_k \ldots \underline{\tilde{\theta}}_1$$
 to obtain formula \eqref{LambdaasR}.  The functional equation follows immediately.
  The shuffle product formula follows from the  standard shuffle product for iterated integrals \cite{Ch}. \end{proof}

When  each $\theta_i= \theta_{f_i}$ is  associated to a  cusp form, the iterated Mellin transforms were previously considered by Manin in \cite{Ma1}.

\subsection{The case with simple poles}  \label{sectWithPoles}

\begin{defn} Let $\theta_1,\ldots, \theta_r$ be as in the beginning of the section.
Define the iterated  regularised Mellin transform to be the iterated integral 
 $$\ML(\theta_1,\ldots, \theta_r; s_1,\ldots, s_r) =  \int_{0}^{\tone_{\infty}} \underline{\theta}_1(s_1) \cdots \underline{\theta}_r(s_r)\ , $$
 where $\tone_{\infty}$ denotes a tangent vector of length $1$ at infinity\footnote{this notation was used in  \cite{MMV}, \S4 
with respect to a coordinate which is  $i$ times the coordinate used here: i.e., the definitions agree after   identifying $\R_{>0}$ with the imaginary axis $i \R_{>0}$  in the upper-half plane. In section \S\ref{sectMultipleXi}  we use both conventions - the meaning will be clear from the context. 
}. The definition is spelled out below. 
  The integral converges for $\mathrm{Re}(s_i)$ sufficiently large. 
\end{defn}
The iterated integral in this case
can be written using a regularisation operator $\mathcal{R}$
$$ \int_{0}^{\tone_{\infty}} \underline{\theta}_1(s_1) \cdots \underline{\theta}_r (s_r)= \int_0^{\infty}  \mathcal{R} \left(  \underline{\theta}_1(s_1) \cdots \underline{\theta}_r(s_r) \right) $$
which  was defined in \cite{MMV}, \S4.6.  The right-hand side of the previous formula is a linear combination of iterated integrals in the $\underline{\theta}_i^0$ and $\underline{\theta}_i^{\infty}$. 
Formally, the operator $\mathcal{R}$ satisfies   $$\mathcal{R}\left( \underline{\theta}_1 \ldots  \underline{\theta}_r  \right) = \sum_{i=1}^r (-1)^{r-i}   \left(( \underline{\theta}_1 \underline{\theta}_2 \ldots  \underline{\theta}_{i-1})  \sha (\underline{\theta}_r^{\infty} \underline{\theta}_{r-1}^{\infty} \ldots \underline{\theta}_{i+1}^{\infty}) \right) . \,   \underline{\theta}_i^{0} $$
where $\sha$ denotes the shuffle product  and $.$ denotes concatenation. 
To make sense of this, one should work in a tensor algebra of differential forms (\cite{MMV}, \S4.6).

\begin{example}
The regularisation operator satisfies 
\begin{eqnarray}\label{examplesofR}
\mathcal{R} ( \underline{\theta}_1 ) &= & \underline{\theta}_1^0  \\
\mathcal{R}( \underline{\theta}_1 \underline{\theta}_2 )  &= & \underline{\theta}_1 \underline{\theta}_2^0 - \underline{\theta}_2^{\infty} \underline{\theta}_1^{0} \nonumber\\ 
\mathcal{R}(\underline{\theta}_1 \underline{\theta}_2 \underline{\theta}_3)  
& = &  (\underline{\theta}_3^{\infty} \underline{\theta}_2^{\infty}) \underline{\theta}_1^0 - ( \underline{\theta}_1 \sha \underline{\theta}_3^{\infty} ) \underline{\theta}_2^0  + 
 (\underline{\theta}_1 \underline{\theta}_2 )\underline{\theta}_3^0
  \nonumber \\ 
&=  &\underline{\theta}_1 \underline{\theta}_2 \underline{\theta}_3^0 - \underline{\theta}_1 \underline{\theta}_3^{\infty} \underline{\theta}_2^0 - \underline{\theta}_3^{\infty} \underline{\theta}_1 \underline{\theta}_2^0 + \underline{\theta}_3^{\infty} \underline{\theta}_2^{\infty} \underline{\theta}_1^0  \nonumber
\end{eqnarray} 
One has the recursive formula (\cite{MMV}, proof of lemma 4.8) 
$$\mathcal{R}(\underline{\theta}_1 \underline{\theta}_2 \ldots \underline{\theta}_n)  =  \underline{\theta}_1 \mathcal{R}(  \underline{\theta}_2 \underline{\theta}_3 \ldots \underline{\theta}_n )  -  \underline{\theta}^{\infty}_n \mathcal{R}(  \underline{\theta}_1 \underline{\theta}_2\ldots \underline{\theta}_{n-1} )\ .$$
\end{example}

Proposition \ref{prop: propertiesnopoles} has the following variant in the case of poles.
\begin{thm}  The formula \eqref{LambdaasR} holds, where we now define
$$ R(\theta_1,\ldots, \theta_r; s_1,\ldots, s_r) = \int_1^{\tone_{\infty}} \underline{\theta}_1(s_1) \ldots \underline{\theta}_r(s_r)\   . $$ The right-hand side is  the  regularised iterated integral
\begin{equation}\label{dag} \int_1^{\tone_{\infty}} \underline{\theta}_1 \ldots \underline{\theta}_r  = \sum_{i=0}^r (-1)^{r-i}  \int_1^{\infty} \mathcal{R} (\underline{\theta}_1  \ldots \underline{\theta}_i ) \int_0^1 \underline{\theta}_r^{\infty}  \underline{\theta}_{r-1}^{\infty}\ldots \underline{\theta}_{i+1}^{\infty}\ .
\end{equation}
The  integrals  of $\mathcal{R} (\underline{\theta}_1 \ldots \underline{\theta}_i)$ from $1$ to $\infty$  in the right-hand side of this formula converge for all $s_i \in \C$ and are holomorphic. The iterated integrals
$$\int_0^1 \underline{\theta}_r^{\infty} \underline{\theta}_{r-1}^{\infty}\ldots \underline{\theta}_{i+1}^{\infty}$$
 can be interpreted  geometrically as integrals in the tangent space  at the point $\infty$ of Riemann sphere \cite{MMV}, \S4. 
They can be computed explicitly since the $\theta_i^{\infty}$ are polynomials in $t$. In particular, they 
define rational functions in the $s_i$
with simple poles along finitely many hyperplanes of the following type:
\begin{eqnarray} \label{Polehyperplanes}
  s_{i}  + \ldots +s_{r-1} + s_r & =  &   \alpha_{i}   
\end{eqnarray}
where $\alpha_i \in \C$. 
 It follows that $\Lambda(\theta_1,\ldots, \theta_r;s_1,\ldots, s_r)$ admits a meromorphic continuation to $\C^r$ with poles along  \eqref{Polehyperplanes}   and  their images under the transformation 
$$s_i \  \mapsto \  w_{r+1-i}- s_{r+1-i}  \qquad \ , \qquad  i=1,\ldots, r\ .$$
The functional equation holds  by the symmetry of equation \eqref{LambdaasR}. \end{thm}

\begin{proof} As  for proposition  \ref{prop: propertiesnopoles}, using the properties of tangential base points at infinity, which are  similar to those of ordinary iterated integrals (see \cite{MMV}, \S4).
\end{proof}

\begin{rem} Using the above formulae, one can  express the residues of $\Lambda(\theta_1,\ldots, \theta_r; s_1,\ldots, s_r)$ in terms of functions of the same type, but with  smaller values of $r$.  
\end{rem}

Because of the exponential decay of $\theta^{\infty}$ at infinity, the formulae above converge extremely fast, and can be  highly effective for numerical computations.

\subsection{General case}
The above definitions are easily modified to encompass the case of motives $M$ which are not self-dual. One replaces the functional equation for $\theta=\theta_M$ by
$$\theta\left(\frac{1}{t}\right) = \varepsilon \,t^w \check{\theta}(t)$$
where $\check{\theta}=\theta_{M^{\vee}}$ is associated to the dual motive. The functional equation becomes
$$\Lambda(\theta_1,\ldots, \theta_r; s_1,\ldots, s_r) = \varepsilon_1\ldots \varepsilon_r \Lambda(\check{\theta}_r, \ldots, \check{\theta}_1; w_r-s_r,\ldots, w_1-s_1)\ .$$
In this manner, one can define, for example,  mixed Dirichlet $L$-functions, and so on.

\section{Examples} 

\subsection{Length 1.}  
 By definition,
$$\Lambda(\theta; s) = \int_{0}^{\tone_{\infty}}  \underline{\theta} (s)
= \int_0^{\infty} \mathcal{R}    \underline{\theta}(s)\ .$$ 
Substituting equation \eqref{examplesofR}, we find that 
\begin{equation} \label{Length1Lambda} \Lambda(\theta; s) = \int_0^{\infty} \theta^0 (t)t^{s-1} dt  = \int_0^{\infty}  \left( \theta(t) - \theta^{\infty}(t)\right) t^s \frac{dt}{t} \ ,
\end{equation}
 which  coincides for $\theta=\theta_f$ with  Hecke's formula for the $L$-function of a  modular  form. 
 Furthermore, if $\theta(\frac{1}{t}) = \varepsilon t^w \theta(t)$, then equation \eqref{LambdaasR} reads:
$$ \ML(\theta;  s) = R(\theta; s) + \varepsilon R(\theta; w-s)$$
where, by \eqref{dag},
$$R(\theta;s) = \int_1^{\infty} \mathcal{R} \underline{\theta}(s) -  \int_0^1 \underline{\theta}^{\infty}(s) =\int_{1}^{\infty} \theta^0 (t) t^{s} \frac{dt}{t} - \int_0^1  \theta^{\infty}(t) t^s \frac{dt}{t}\ .$$
 
\begin{example}For the motive $\Q$, we have $\theta_{\Q}^{\infty}=1$, $\varepsilon=1$,  and  hence
$$R(\theta_{\Q}; s) = \int_1^{\infty} \left( \theta_{\Q}(t) -1\right)t^s \frac{dt}{t} - \frac{1}{s}\  .$$
In this way we obtain the classical formula for Riemann's $\xi$-function:
$$\xi(s) = \Lambda(\theta_{\Q};s) =   \int_1^{\infty} \left(\theta_{\Q}(t) -1\right) t^s \frac{dt}{t} +  \int_1^{\infty}  \left( \theta_{\Q}(t)-1\right) t^{1-s} \frac{dt}{t} - \frac{1}{s} -\frac{1}{1-s}\ .$$
\end{example}

\subsection{Length two} We have
$$\Lambda(\theta_1,\theta_2; s_1,s_2) = \int_{0}^{\tone_{\infty}}  \underline{\theta}_1  (s_1) \underline{\theta}_2 (s_2)
= \int_0^{\infty} \mathcal{R}   \left(\underline{\theta}_1(s_1) \underline{\theta}_2(s_2)\right)\ .$$ 
Substituting  \eqref{examplesofR} into this definition,  we find that 
\begin{multline} \label{lambdain2vars} \Lambda(\theta_1 ,\theta_2;s_1,s_2)  = \int_{0\leq t_1 \leq t_2 \leq \infty} \theta_1(t_1) t_1^{s_1-1} \, \theta_2^{0}(t_2) t_2^{s_2-1} \,  dt_1 dt_2    \\
- \int_{0\leq t_1 \leq t_2 \leq \infty} \theta_2^{\infty}(t_1) t_1^{s_2-1} \, \theta^0_1(t_2) t_2^{s_1-1}   \, dt_1 dt_2
\end{multline}
In order to  compute this function, use formula (\ref{LambdaasR}) which reads
$$\Lambda(\theta_1 ,\theta_2;s_1,s_2) =  R(\theta_1,\theta_2; s_1,s_2) 
+ \varepsilon_1 R(\theta_1; w_1-s_1)R(\theta_2; s_2) + \varepsilon_1 \varepsilon_2 R(\theta_2,\theta_1;w_2-s_2,w_1-s_1) \nonumber
 $$
and where
$$R(\theta_1,\theta_2;s_1,s_2) = \int_{1}^{\infty} \mathcal{R} \left(  \underline{\theta}_1 \underline{\theta}_2\right) - 
\int_1^{\infty} \mathcal{R}( \underline{\theta}_1) \int_0^1 \underline{\theta}_2^{\infty} + \int_0^1\underline{\theta}_2^{\infty} \underline{\theta}_1^{\infty}\ .$$
The first integral on the right-hand side is the same as the right-hand side of \eqref{lambdain2vars} except that both lower limits of integration $0$ are replaced by $1$. The integral 
$$\int_0^1 \underline{\theta}_2^{\infty} \underline{\theta}_1^{\infty} = \int_{0\leq t_1\leq t_2\leq 1} \theta_2^{\infty}(t_1) t_1^{s_2-1} dt_1\,  \theta^{\infty}_1(t_2) t_2^{s_1-1} dt_2$$
 is easy to compute since $\theta_1^{\infty}$, $\theta_2^{\infty}$ are simply polynomials.
\begin{example} \label{exDoubleXi}
Consider the  function $\xi(s_1,s_2)=\Lambda(\theta_{\Q}, \theta_{\Q};s_1,s_2)$ which we shall  call the double Riemann $\xi$-function.   It is studied in more detail in \S \ref{sectMultipleXi}.  We have
$$\int_0^1 \underline{\theta}_{2}^{\infty} \underline{\theta}_{1}^{\infty} = \int_{0\leq t_1\leq t_2\leq 1}  t_1^{s_2-1}   t_2^{s_1-1} \, dt_1 dt_2 = \frac{1}{s_2(s_1+s_2)}\  , $$
where $\theta_1=\theta_2= \theta_{\Q}$. 
Putting the  pieces together, we find that 
$$R(\theta_{\Q},\theta_{\Q};s_1,s_2) = \int_{1}^{\infty} \mathcal{R} \left( \underline{\theta
}_{1}  \underline{\theta}_{2}\right) - 
\frac{1}{s_2} \int_1^{\infty} \mathcal{R}(\underline{\theta}_{1}) + \frac{1}{s_2(s_1+s_2)}\ .$$
We deduce from this that $   \ML(\theta_\Q,\theta_\Q;s_1,s_2) $ has poles along 
$$s_1= 1 \ , \ s_2=0   \ , \ s_1+s_2 =0 \ , \ s_1+s_2 =2 \ . $$
One can easily compute  its residues, e.g.
\begin{eqnarray} \mathrm{Res}_{s_1+s_2=0}\,  \ML(\theta_\Q,\theta_\Q;s_1,s_2) & = & s_2^{-1}  \nonumber \\
\mathrm{Res}_{s_2=0}\,  \ML(\theta_\Q,\theta_\Q;s_1,s_2)  &=  & - \ML(\theta_\Q;s_1) \ . \nonumber
\end{eqnarray} 
\end{example}

\subsection{Relation to  Dirichlet series}  \label{sectRelDirichlet}
It is important to note that  multiple $\ML$-functions  are not  expressible as  Dirichlet series in general. 
Consider the case where $r=2$ and $f_1,f_2$ are modular forms with Fourier expansions
$$f_1 = \sum_{n\geq 0} a_n q^n \quad \ , \quad f_2 = \sum_{n \geq 0} b_n q^n \ .$$
For the time being,  let us assume that $a_0=b_0=0$ for simplicity. 
Then  by making the change of variables $t_1=xy, t_2=y$ in the definition
$$\ML(\theta_{f_1}, \theta_{f_2};s_1,s_2) = \int_{0\leq t_1 \leq t_2 \leq \infty}  f_1( i t_1) f_2( it_2)  t_1^{s_1-1} t_2^{s_2-1} dt_1 dt_2 $$
and expanding, we obtain (for  $\mathrm{Re} (s_1), \mathrm{Re} (s_2) $ sufficiently large)
$$\ML(\theta_{f_1}, \theta_{f_2};s_1,s_2) = \sum_{m,n\geq 1}  a_m b_n  \int_{0\leq x\leq 1} x^{s_1-1} dx  \int_{0}^{\infty} e^{-2 \pi (mx+n) y} y^{s_1+s_2-1}dy$$
The right-hand integral is  a  simple Mellin transform:
$$\int_{0}^{\infty} e^{-2 \pi (mx+n) y} y^{s_1+s_2-1}dy  \quad =  \quad (2\pi )^{-s_1-s_2}  \Gamma(s_1+s_2) \frac{1}{(mx+n)^{s_1+s_2}}\ .$$
It follows that 
\begin{equation} \label{Lambdanotdirichlet}  \ML(\theta_{f_1}, \theta_{f_2};s_1,s_2)  =   (2\pi )^{-s_1-s_2} \Gamma(s_1+s_2) \sum_{m,n\geq 1} a_m b_n \int_0^1 \frac{x^{s_1}}{(mx+n)^{s_1+s_2}} \frac{dx}{x}
\end{equation} 
The hypergeometric integrals  
$$\int_0^1 \frac{x^{s_1}}{(mx+n)^{s_1+s_2}} \frac{dx}{x}$$
reduce to rational functions in $m,n$ when $s_1, s_2$ are integers, but not otherwise. 
\begin{lem} If $p, m,n>0$ are  integers, then
$$\frac{\Gamma(s+p)}{(p-1)!} \int_0^1 \frac{x^{p}}{(mx+n)^{p+s}} \frac{dx}{x} \ = \ \frac{ \Gamma(s)}{m^p n^s} \ - \  \sum_{r=0}^{p-1} \frac{1}{r!}\frac{\Gamma(s+r)}{m^{p-r}(m+n)^{s+r}}\ .$$
\end{lem} 
Therefore, consider the Dirichlet series
$$D(f,g; k,s) = \sum_{m, n\geq 1}  \frac{a_m b_n }{m^k (m+n)^s}$$
which arose (in the case of cusp forms) in a similar form in \cite{Ma1}, \S3  and write
$$\mathbb{D}(f,g;k,s) = (2\pi)^{-k-s} \Gamma(k)  \Gamma(s)  \, D(f,g;k,s)\ .$$
When the first argument is a  fixed integer, the  double $\ML$-function  $\Lambda(\theta_f,\theta_g)$  is a  linear combination of  Dirichlet series, but this is not true  in general:

\begin{lem} Let $p\geq 1$ be an integer, and let $f,g$ be modular forms for the full modular group which are not necessarily cuspidal. Then 
$$\ML(f,g;p,s) = \ML(f,p) \ML(g,s)  +  \frac{a_0}{p} \ML(g,s+p)  - \frac{b_0}{s} \ML(f,s+p) 
- \, \sum_{r=0}^{p-1}   \binom{p-1}{r} \, \mathbb{D}(f,g; p-r,s+r) 
$$
\end{lem}
\begin{proof} Use the definition of the double $\ML$-functions as  an iterated integral
$$\Lambda(f,g;p,s) = \int_0^{\infty} \theta_f^0 \theta_g^0 +    \theta_f^{\infty} \theta_g^0 - \theta_g^{\infty} \theta_f^0 \ ,  $$
expand as above and apply  the previous lemma.
\end{proof} 
By varying $p$ in the previous lemma we can express each 
$\mathbb{D}(f,g; p ,s )$ as a linear combination of  products of  two single $\ML$-functions, 
and  double $\ML$-functions in which the first argument is constant. 
It follows that  $\mathbb{D}(f,g;p,s)$ admits a memorphic continuation to $\C$.

\section{Variants} Some naturally occurring Dirichlet series require a  slight modification of \S  \ref{sectThetadef}.

  \subsection{Examples}

\begin{example}
 One can apply the definition of an $L$-function  to mixed motives. For instance, for a direct sum $M\oplus N$, the definitions give 
 $$L(M \oplus N; s) = L(M;s) L(N;s)\ .$$
We are thus led to consider products of Dirichlet series.  We  require that their  completed $L$-functions  have distinct poles and that their product satisfies a functional equation (this happens, for example: if $M$, $N$ are pure and have the same weight, or  if $N\cong  M^{\vee}(n)$ for some $n\in \Z$). 
\end{example} 

\begin{example}  \label{exampleramified} If one considers  motives $M$ over rings of  integers,
it can happen that  $L(M;s)$ and $L(\gr^W M;s)$ differ by a finite number of Euler factors \cite{SchollLvalues}.

For $p$ prime, consider $M=H^1(\Pro^1\backslash \{0,\infty\}, \{1,p\})$ which is an  extension 
$$0 \To \Q \To M  \To \Q(-1)\To 0$$
(e.g., in the category $\MT(\Z[1/p])$ of mixed Tate motives over $\Q$   ramified only at $p$). One of its periods is $\log p$. Its $L$-function is 
$$L(M;s) = (1-p^{-s}) \,\zeta(s) \zeta(s-1)$$
which  differs from $\zeta(s) \zeta(s-1)  = L(\Q \oplus \Q(-1);s)$  by a single Euler factor.

 Similar examples  include Dirichlet's $\lambda$ and $\eta$ functions:
\begin{eqnarray} \label{Dirichetlambdaeta} L_\lambda(s) & =&  \left(1- 2^{-s}\right) \zeta(s) \  = \  \sum_{n\geq 1 } \frac{1}{(2n+1)^s}    \\
L_\eta(s)  &= & \left(1- 2^{1-s}\right) \zeta(s) \  = \  \sum_{n\geq 1 } \frac{(-1)^n}{n^s} \nonumber
\end{eqnarray}
The values of $L_\eta(s)$ at positive integers are   called Euler sums, and  are periods of simple  extensions of mixed Tate motives over $\Z[\frac{1}{2}]$, i.e.,  ramified at the prime $2$.
\end{example} 

\begin{example} \label{ExampleEis} (Eisenstein series.) Hecke's formula for  $L$-functions applies to all modular forms,  not just cusp forms. 
Consider the Eisenstein series
$$\G_{2k} = - \frac{b_{2k} }{4k} + \sum_{n\geq 1}  \sigma_{2k-1}(n)\, q^n $$
where $\sigma$ denotes the divisor function. It defines a modular form of weight $2k$ and level one, for all $k\geq 2$. 
Let $\G_{2k}^0 = \G_{2k} +   \frac{b_{2k} }{4k}$ denote the Eisenstein series with its zeroth Fourier coefficient removed. Following Hecke, one defines
$$\Lambda(\G_{2k}; s) = \int_0^{\infty}  \G_{2k}^0 (\tau) \tau^s \frac{d\tau}{\tau}$$
which converges for $\mathrm{Re}(s)>2k$.   It coincides with our definition \eqref{Length1Lambda}. It admits a meromorphic continuation to $\C$ and satisfies the functional equation 
$$\Lambda(\G_{2k}; s) = (-1)^k \Lambda(\G_{2k}; 2k-s)  \  . $$
It can be written  $\Lambda(\G_{2k}; s)  =  (2\pi)^{-s}  \Gamma(s)\,  L(\G_{2k};s)$ where the $L$-series
 $$L(\G_{2k}; s) \ =  \ \sum_{n\geq 1} \frac{\sigma_{2k-1}(n) }{ n^s} \  =  \  \zeta(s) \zeta(s-2k+1)$$
 factorises as a product of zeta functions.
 One might  think that the associated motive is $\Q(0) \oplus \Q(1-2k)$ but the gamma factors  do not agree.
  Indeed, the completed $L$-function of the latter is $\xi(s) \xi(s-2k+1)$ but in fact the completed $L$-function of the Eisenstein series is different:
   \begin{equation} \label{LambdaG2kasxifunctions} \Lambda(\G_{2k}; s) = \left(\frac{ (s-1)(s-3) \ldots (s-2k+1)}{ 2 (2 \pi)^{k}} \right) \, \xi(s) \xi(s-2k+1)\ .
 \end{equation}
 The  polynomial prefactor in brackets   plays an important role.  The critical values of $\Lambda(\G_{2k};s)$ will be defined to be the integers $s=1,\ldots, 2k-1$. 
    In this example it is not   the $L$-function which has changed but rather the completed $L$-function.
 \end{example}

\subsection{Reminder on Mellin transforms}
Recall that the Mellin transform 
$$\Mel(f)(s) = \int_0^{\infty} f(t) t^s \frac{dt}{t}$$
for a suitable continuous function $f$, 
satisfies the formal properties:
 \begin{eqnarray} \label{Mellinproperties}
 \Mel( t f)(s) & = & \Mel(f) (s+1) \\ 
 \Mel (f (tn)) & = & n^{-s}  \Mel(f(t))  \nonumber \\
 s\, \Mel(f) & = & - \Mel( f' t) \nonumber \\
 \Mel(f_1\star f_2) & = &  \Mel( f_1) \Mel(f_2) \nonumber 
  \end{eqnarray} 
where the convolution is defined by 
$$(f_1 \star f_2 )(t) = \int_0^{\infty} f_1\left( \frac{t}{x}\right) f_2(x) \frac{dx}{x}$$
and  all integrals are assumed to  converge. We shall apply these operations to  the functions $\theta^0$ where $\theta$  satisfies  \S\ref{sectIteratedMellin},  (1) and (2).   They   do not necessarily preserve these properties.

\begin{rem}
Since the gamma factor of a motivic $L$-function is a product of gamma  functions $\frac{1}{2}\Gamma(\frac{s}{2})$, its inverse Mellin transform can be generated  from a single function  $e^{-t^2}$ using the operations \eqref{Mellinproperties}.
 Computing the expansions of multiple $L$-functions (as in \S\ref{sectRelDirichlet}, for example),   requires us to study the   larger class of functions   generated by  $e^{-t^2}$ under  these operations  together with the extra operation of   taking  indefinite iterated integrals.
 \end{rem}
  
 \subsection{Operations on theta functions}
 The previous  examples suggest the following natural operations on $L$-functions of motives.
  \begin{enumerate} [label=(\textbf{L\arabic*})]
 \item (Tate twists). Replacing $M$ by $M(n)$ shifts the argument:
$$\Lambda(M(n), s) = \Lambda(M, s+ n)$$

  \item (Local factors).  Changing one or more Euler factors amounts to  multiplying $\Lambda(M;s)$ by a polynomial in $n^{-s}$, for some $n\in \N$.

 \vspace{0.1in}
 \item (Gamma factors).  Modifying the gamma factors, via $\Gamma(s+1) = s\Gamma(s)$, can be  encoded by multiplying $\Lambda(M;s)$ by a polynomial in $s$. 
 
 \vspace{0.1in}

 \item (Direct sums). Taking direct sums corresponds to multiplication:
$$\Lambda(M\oplus M';  s) = \Lambda(M;s) \, \Lambda(M';s)$$
\end{enumerate} 

Operations $(\mathbf{L1})-(\mathbf{L3})$ preserve the rank of $\gamma(s)$,  but $(\mathbf{L4})$ does not.   
\\
 
We wish to apply the above operations in such a manner that  our initial assumptions on the $L$-functions and, most importantly, their   functional equations are respected. 
\vspace{0.1in}

By  \eqref{Mellinproperties}, these operations are generated on  inverse Mellin transforms   by:

\begin{enumerate} [label=(\textbf{T\arabic*})]
\vspace{0.1in}

\item (Multiplication by $t^{\pm}$).  $\theta(t) \mapsto t^{\pm } \theta(t)$. 
\vspace{0.1in}

\item (Rescaling). $\theta(t)\mapsto \theta(n t)$. 
\vspace{0.1in}

\item (Differentiating).  $\theta(t) \mapsto   - t \theta'(t).$
\vspace{0.1in}
\item (Convolution product). $\theta_1,\theta_2 \mapsto \theta_1\star \theta_2$
\vspace{0.1in}
\end{enumerate}

 \subsection{New theta functions from old} \label{sectNewTheta} 
Starting with the inverse Mellin transforms of  a set of motives \S\ref{sectThetadef}, we shall allow ourselves to generate  new theta functions using  $(\mathbf{T1})-(\mathbf{T4})$, provided that they preserve the properties  (1) and (2) of \S\ref{sectIteratedMellin}. For $(\mathbf{T4})$ this means that  $(\theta_1 \star \theta_2)^0 = \theta_1^0 \star \theta_2^0$ will be  a convolution integral, and $(\theta_1 \star \theta_2)^{\infty}$  must be  a polynomial. It  is uniquely determined from $(\theta_1 \star \theta_2)^0$ by the inversion formula.

For example, one might multiply by $t^{-1}$ in $(\mathbf{T1})$ but only on condition that $t^{-1} \theta^0(t)$  is still a polynomial. Similarly, operation  $(\mathbf{T3})$ does not in general respect the inversion relation, but can do  when combined with the other operations. 
For example,  combining $(\mathbf{T1})$ and $(\mathbf{T3})$ yields a differential operator
 $$D_w \theta = - (w+1) t \theta'  - t^2 \theta''$$
 which preserves the inversion formula $\theta(t^{-1}) = \pm t^w \theta(t)$  in degree $w$.  
By Mellin transform it corresponds to    multiplication by the factor $s(w-s)$.

 The operations described above can be codified by working in the $(\N \times \Z/2\Z)$-graded vector space  $\Theta(S)$ whose elements in 
 degree $(w,\varepsilon)$ satisfy the inversion and growth properties \S\ref{sectIteratedMellin}, (1), (2), and  which is generated by the inverse Mellin transforms  of a set $S$ of motives.  The space $\Theta(S)$ is stable under certain  combinations of operators  $(\mathbf{T1})-(\mathbf{T4})$.

 \begin{rem} \label{RemMult} (Multiplicative structure).  There is another operation on the graded space of  theta functions, which is multiplication.  Its counterpart for $L$-functions is a convolution operation which  is quite alien from an arithmetic  perspective. 
 
 Nonetheless,    iterated integrals essentially  subsume multiplication, since 
 $$\int \theta_1'  dt \, \theta_2 dt =  \int \left(\theta_1 \theta_2\right) dt$$ 
 where the integral on the left is a double iterated integral, and the one on the right is a single  integral of the product of $\theta_1$ and $\theta_2$. 
 This gives a first hint  as to why one can  find unexpected periods arising out of multiple $L$-functions associated to motives. 
  \end{rem}

\section{Multiple $\ML$-values and periods} \label{sectMultipleML}

\subsection{Totally critical values}
Starting from a finite  collection of motives and their inverse Mellin transforms, we can generate a space of theta functions  following \S\ref{sectNewTheta}, and consider the associated multiple $L$-functions of  \S\ref{sectIteratedMellin}. The next two paragraphs give examples where the values 
 $\ML(\theta_1,\ldots, \theta_r;n_1,\ldots, n_r)$ are related to interesting periods.
In all these examples, the values of $n_i$ are   integers of the following form. 
\begin{defn}  Call $n_1,\ldots, n_r \in \Z $  \emph{totally critical} for $\theta_1,\ldots, \theta_r$  if, for 
 each $i$,  $n_i$ is critical for $\theta_{i}$.
\end{defn} 
Since we have not defined critical values for arbitrary theta functions, this definition has  limited value.  However, when $\theta_M$ is the inverse Mellin transform of a simple motive $M$,  the standard definition \S\ref{sectConjL}  applies.
When $\theta = \theta_f$ is associated to  a  modular form $f$ of weight $w\in \N$,  we say that  the critical values of $\theta_f$ are $1 \leq n \leq w-1$. This  agrees with the usual definition   for  cusp forms, but, strictly speaking,   falls  outside its  scope when $f$ is  an Eisenstein series (example \ref{ExampleEis}).

\begin{example}
Suppose  that $\theta_i=\theta_{f_i}$, where $f_i$ are modular forms of integer weight $w_i\geq 2$ for a congruence subgroup of $\SL_2(\Z)$. The totally critical values
$$\ML(\theta_1,\ldots, \theta_r;  n_1,\ldots, n_r) \quad \hbox{ for } \quad  0 <  n_i <  w_i \ ,$$ 
are called totally holomorphic multiple modular values and  can  be interpreted  as periods of the relative completion \cite{HaGPS} of the fundamental groupoid of the underlying modular curve, with possible tangential base points.  Examples  of  level one were given  in the first half of this talk \cite{BrIhara} and illustrate some of the phenomena which can arise. 
\end{example}

\subsection{Remarks}
It is tempting to  ask if    every totally critical value  of a multiple $L$-function $\ML(\theta_1,\ldots, \theta_r; n_1,\ldots, n_r)$   is a period. 
Considerable  caution is required 
in the case where  the $\theta_i$ are  associated to motives of  higher rank $>2$ (or have varying  ranks),
since I currently know of no examples where this is either true or false for    $r\geq 2$. 

In low ranks, examples suggest that if each $\theta_i$ is the theta  function  associated to a  simple motive, then the totally critical values 
$\ML(\theta_1,\ldots, \theta_r; n_1,\ldots, n_r)$   are non-effective periods  of coradical (unipotency) filtration  $\leq r-1$ (see for example, theorem 22.2 of \cite{MMV}). If true,  this would be  consistent with Deligne's conjecture in the classical case  $r=1$, since a period of coradical filtration zero is a  pure period \cite{BrNotes}. Note that  the $L$-function of an Eisenstein series, as in example \ref{ExampleEis},  has a critical value which is an odd zeta value and hence has coradical filtration $1$.

One might  think that  $\ML(\theta_1,\ldots, \theta_r;n_1,\ldots, n_r)$ is a period of an iterated extension of tensor products of the motives one intially starts  with. This is  not quite  the case because of the appearance of secondary or `convolution' periods which are not unrelated to the multiplicative structure on theta functions (remark \ref{RemMult}).  However, \cite{BrIhara} gives  many examples where, by taking appropriate linear combinations of totally critical values,  it is   possible to remove these additional periods. The Ihara-Takao  equation and its variants are  examples of this phenomenon.

If we allow ourselves  to speculate even further, then we  might hope that a certain class of periods (defined using the Hodge filtration) of \emph{all possible} iterated extensions of motives of a given type would be expressible in terms of multiple $L$-values (see theorem \ref{thmMZVasML} for an example). 
A  simple  situation where this could  potentially occur is captured by the following:

\begin{conj}
Let $\mathcal{E}$ be an extension of $\Q$ by $M$ in a suitable category of realisations of motives\footnote{for instance the Tannakian category generated by the Betti and de Rham realisations of $H^n(X\backslash A, B \backslash (A\cap B))$, where $X$ is smooth projective over $\Q$, and $A,B$ are a simple normal crossing divisor  with no common components. } over $\Q$,  where $M$ is of rank 2 and has weight $m\leq - 2$, i.e.,  
$$0 \To M  \To \mathcal{E} \To \Q \To 0\ . $$
 Assume $ M_{dR}$ is not of type $(m,m)$, so  $F^m \mathcal{E}_{dR}$ is two-dimensional.  We conjecture that
 $$\mathrm{comp}_{B,dR} \, \left( \textstyle{\bigwedge^2} F^m \mathcal{E}_{dR} \right) \quad \subseteq \quad   \textstyle{\bigwedge^2} \mathcal{E}_B\otimes_{\Q} R$$
 where $R$ is the  vector space over $\overline{\Q}[2\pi i
]$  generated by the multiple  $\ML$-values of length at most two generated from the theta functions of the  objects $\Q$ and $M$. 
\end{conj}

In suitable bases of $\mathcal{E}_B$ and $\mathcal{E}_{dR}$, the comparison isomorphism  has the period matrix 
$$\begin{pmatrix} \eta^+  & \omega^+ & \alpha \\ 
i \eta^- & i \omega^-  & i \beta \\
0 & 0 & 1 
\end{pmatrix} $$
where the top left hand $2\times 2$ square matrix is a period matrix for $M$ and $\eta^{\pm}, \omega^{\pm}, \alpha, \beta$ are real.  The image of the  subspace $ F^m \mathcal{E}_{dR} \subset \mathcal{E}_{dR}$ is spanned by the right-most two columns. 
The conjecture states that $i( \omega^+  \beta - \alpha  \omega^-)$, $\omega^+$, $i \omega^-$ are expressible in terms of single or  double $\ML$-values (for the latter two numbers, this  is Deligne's conjecture). Note that Beilinson's conjecture already relates  the quantity $i \beta$ to  the $L$-function of $M$. The  content of the above conjecture is to explain  $\alpha$, which is only well-defined up to addition of a rational  (since in the above period matrix, one can add  a rational multiple of the third row  to the first). See \cite{BrIhara}, \S7 for an example.

 \subsection{Multiple Dirichlet series}  It is important to point out that the integer values of  multiple $L$-functions are in general  different from  values  of 
`multiple Dirichlet series',
e.g.,  $$\sum_{n_1, \ldots, n_r \geq 1}  \frac{a^{(1)}_{n_1} \ldots a^{(r)}_{n_r}} {n_1^{k_1}(n_1+n_2)^{k_2} \ldots (n_1+\ldots+ n_r)^{k_r}}$$
 where $a^{(i)}_m$ are the coefficients of  Dirichlet series.  Multiple $L$-values are linear combinations of
 $$ \sum_{n_1, \ldots, n_r \geq 1}    a^{(1)}_{n_1} \ldots a^{(r)}_{n_r}  \,  H(n_1,\ldots, n_r)$$
 for some `binding functions' $H$ which are hypergeometric. 
  However, 
 there exist examples where multiple Dirichlet series are in fact totally critical values of multiple $L$-functions    (see, for example,  theorem \ref{thmMZVasML}).   Similarly,  Horozov has defined  multiple Dedekind zeta values associated to  number fields \cite{Horozov}. They  do not appear to be related to the multiple $L$-functions defined in \S\ref{sectMultipleML}, but this does not rule out a  connection between the two kinds of objects.

\section{Example: modular forms of weight  two}  
Let  $N\geq 1$ and let 
$f_1,\ldots, f_r$
be modular forms of weight two  for $\Gamma \leq \Gamma_0(N)$ of finite index.
For simplicity, suppose that they have no poles at the cusps  $\tau = \{0, i \infty\}$.  We can 
assume that the $f_i$ are eigenfunctions for the  involution $\tau\mapsto -\frac{1}{N\tau}$, i.e., 
$$f_i \left( - \frac{1}{N\, \tau  }  \right)  = -  \varepsilon_i  N\,  \tau^2 f_i(\tau)\qquad \hbox{ for all } 1\leq i \leq r\ .$$
The  associated theta functions  are   $\theta_i (t)=   f ( i  t  N^{-1/2} )$ and   satisfy   $\theta_i(t^{-1}) =\varepsilon_i  t^2 \theta_i(t)$. By assumption, the $\theta_i^{\infty}$ vanish for all $i$, so there will be no need to regularise any integrals.   Note that the $f_i$ are allowed to have poles at other cusps besides $0$, $i\infty$. 
 The construction of \S\ref{sectIteratedMellin} gives rise to 
multiple $\ML$-functions satisfying 
 $$\Lambda(\theta_1,\ldots, \theta_r; s_1,\ldots,s_r ) = \varepsilon_1 \ldots \varepsilon_r \, \Lambda(\theta_r,\ldots, \theta_1; 2-s_r, \ldots, 2-s_1)\ ,$$
 and which have a unique totally critical (`central') value
 $$\Lambda(\theta_1,\ldots, \theta_r; 1,\ldots, 1)   \ .$$
 It is related to periods  as follows. Let $\overline{X}_{\Gamma}$ denote the  corresponding compactified modular curve.
 It is convenient to work over a number field  $k\subset \C$ which  contains  $\sqrt{N}$, such   that    $\overline{X}_{\Gamma}$, its cusps,  and we assume,  $ f_1,\ldots, f_r$ are all defined over $k$.
 Let $X_{\Gamma}= \overline{X}_{\Gamma} \backslash D$ where $D$ is the union of   all cusps except those corresponding to  $\tau =\{0, i \infty\}$. When  the $f_i$ are Hecke eigenforms,  the $\theta_i$ are the inverse Mellin transforms of the $L$-functions of the associated submotives  $M_{f_i} \subset H^1(X_{\Gamma}; k)$ which are   eigenspaces for the action of Hecke operators \cite{Scholl} and have rank $1$ or $2$.

 \begin{thm}  The central values 
 $ \pi^r \Lambda(\theta_1,\ldots, \theta_r; 1,\ldots, 1)$ are  periods of a subquotient of the affine ring of the unipotent fundamental groupoid of $X_{\Gamma}/k$ with basepoints given by the image of $0$ and $i\infty$.  More precisely, they are effective periods  of weight and  Hodge filtration $r$  of
 an iterated extension of Tate twists of the pure objects
$H^1(\overline{X}_{\Gamma};k)^{\otimes j}$, where $0\leq j \leq r.$ 
 \end{thm}
 
 \begin{proof} Since all $\theta_i^{\infty}$ vanish, we have the convergent iterated integral representation
 $$(2\pi)^r \Lambda(\theta_1,\ldots, \theta_r; 1,\ldots, 1) =  (2\pi)^r \int_0^{\infty} f_1(i N^{-1/2} t)dt  \ldots f_r(i N^{-1/2} t) dt\ .$$ 
 By changing variables $\tau= i N^{-1/2} t$,  this equals the  iterated integral
 $$ (- \sqrt{N})^{r}  \int_{0}^{i\infty}   (2\pi i  f_1(\tau) d\tau)  \ldots    (2\pi i  f_r(\tau) d\tau)\ . $$
Let  $  \omega_j = 2\pi i  f_j(\tau) d\tau$ for all $j$. Since  $\omega_j \in \Gamma(\overline{X}_{\Gamma},  \Omega^1_{\overline{X}_{\Gamma}}(\log D))$, this integral is proportional by an element in $k^{\times}$ to the iterated integral  on $X_{\Gamma}(\C)$
$$\int_{\gamma} \omega_1 \ldots \omega_r$$
where $\gamma$ is a  geodesic path between the cusps $x_0, x_{\infty}$ defined by the images of  $0, {i\infty}$. It is  a period of the unipotent fundamental groupoid since
 $$\omega_1 \otimes \ldots \otimes  \omega_r \quad \in \quad  W_rF^r \Or( \pi_1^{dR}( X_{\Gamma}/k;x_0,x_{\infty}))$$
 and since  $\gamma$ defines an element 
 $\gamma^B \in \pi_1^{B}(X_{\Gamma}/k;x_0,x_{\infty})(\Q)$.  By Beilinson's geometric construction  \cite{DeGo}, this iterated integral can be interpreted, in any suitable category of realisations,  as a period of  the relative cohomology group
 $H^r(X_{\Gamma}^r, Y)$
 where $Y$ is the normal crossing divisor  (to be slightly modified in the case $x_0=x_{\infty}$, see \cite{DeGo}, \S3) 
 $$\left(\{x_0\} \times  X_{\Gamma}^{r-1} \right) \ \cup \  \Delta_{1,2} \cup \ldots \cup \Delta_{r-1,r} \  \cup \   \left(X_{\Gamma}^{r-1} \times \{x_{\infty}\} \right)$$
 and $\Delta_{i,j}$ denotes the  diagonal in $X_{\Gamma}^r$ where the $i$th and $j$th coordinates coincide.  The statement about extensions follows easily from the  standard  cohomology spectral sequence for relative cohomology which satisfies
 $$E_1^{pq} = \bigoplus_{i_1,\ldots, i_p} H^q(Y_{i_1} \cap \ldots \cap Y_{i_p}) $$
  where  the $Y_i$ are irreducible  components of $Y$.    Since $Y_{i_1} \cap \ldots \cap Y_{i_p}$ is isomorphic to a product $X_{\Gamma}^{r-p}$, it follows from the K\"unneth formula that  $H^r(X_{\Gamma}^r, Y)$ is an iterated extension of tensor products of the objects $H^i(X_{\Gamma})$. These are pure Tate for $i\neq 1 $.  Furthermore,  $H^1(X_{\Gamma})$ is itself an extension of the required type by Gysin:  $$0 \To H^1(\overline{X}_{\Gamma}) \To H^1(X_{\Gamma}) \To H^0(D) (-1)\To H^2(X_{\Gamma})$$
since 
  $H^0(D)(-1)$ is Tate over $k$.    \end{proof}
 
 We claim that the periods in the theorem have coradical filtration $\leq r$, and furthermore if one of the $\theta_i$'s is cuspidal, then this drops to $ \leq r-1$. 

\begin{rem} One can get rid of the $(\sqrt{N})^r$ occuring in this proof by working with $L^*(\theta_1,\ldots, \theta_r)$ instead of $\ML(\theta_1,\ldots, \theta_r)$ via \eqref{LstarversusLambda}. Since algebraic numbers are periods this does not affect the gist of the theorem.  Similar remarks apply in  the next paragraphs, but   it is more  convienent to state our results in terms of the functions  $\ML$.  
\end{rem}

\section{Example: mixed Tate motives over $\Z$}  \label{sectMTZ}

Mixed Tate motives over the integers are one of the few classes of mixed motives whose periods are completely known \cite{MTZ}.
Indeed, as discussed in the first half of this talk \cite{BrIhara},  they are $\Q[(2\pi i)^{-1}]$-linear combinations of multiple zeta values
$$\zeta(n_1,\ldots, n_r) = \sum_{1\leq k_1 <\ldots< k_r}  \frac{1}{k_1^{n_1} \ldots k_r^{n_r}}  \ ,$$
where $n_r\geq 2$.   In this paragraph, we shall show that they are  totally  critical values of multiple $L$-functions derived from  the trivial motive $\Q$.
This is  not the only way in which this could be achieved, but possibly one of the simplest. 

\subsection{A pair of $L$-functions} 
One might hope that  the periods of all mixed Tate motives over $\Z$  arise as  values of the multiple Riemann $\xi$-function. 
I do not know if this is the case. One problem with this function is that there is no integer point which lies within the critical box $0<s<1$ at which it can be evaluated. Instead, we could start with the motivic $L$-function associated to  a direct sum of Tate motives
$\Q(0) \oplus \Q(-1)$.  It is  a product of two Riemann $\xi$-functions $\xi(s) \xi(s-1)$,  whose     functional equation is $s\mapsto 2-s$, but  has a double pole at $s=1$.   Consider  the following variants which have modified Euler factors at the prime $2$, and no pole at $s=1$
\begin{eqnarray}
\Lambda_{+}(s) & = &  \frac{2}{\pi}  (s-1) \, \left(2^s - 2^{2-s}\right)  \xi(s) \xi(s-1)  \nonumber \\
\Lambda_{-}(s)&  =&    - \frac{12}{\pi}  (s-1) \, \left(2^{s-1} - 1\right) \left(1-  2^{1-s}\right)  \xi(s)  \xi(s-1) \nonumber  
\end{eqnarray} 
They are generated from $\xi(s)=\ML(\theta_{\Q},s)$ from the operations $(\mathbf{L1})-(\mathbf{L4})$ and satisfy
$$\Lambda_{\pm}(s) = \varepsilon_{\pm}  \Lambda_{\pm}(2-s)$$
where $\varepsilon_+ =1 $ and $\varepsilon_- =-1$.   The point $s=1$ is critical, since
by the duplication formula \eqref{Duplication}, the above functions  can be expressed in the form
$$ \Lambda_{\pm} (s)   =  \pi^{-s} \Gamma(s)  L_{\pm}(s)  $$
where  $L_{\pm}(s)$ are the following Dirichlet series: 
\begin{eqnarray} 
L_+(s) & = &  8\, \zeta(s)   \left( (1-4^{1-s}) \,\zeta(s-1) \right)\nonumber \\
L_-(s)&  = &  -24\,   \left((1-2^{1-s}) \, \zeta(s) \right)  \left((1-2^{1-s}) \, \zeta(s-1)  \right) \ . \nonumber 
\end{eqnarray} 
The second is  a product $L_{-}(s) = -24 \,L_{\eta}(s) L_{\lambda} (s-1)$ of Dirichlet's functions \eqref{Dirichetlambdaeta}. 
Denote their inverse Mellin transforms by $\theta_\pm(t)$. They satisfy
$$\theta_{\pm} \left( t^{-1}  \right) = \varepsilon_{\pm} t^2 \theta_{\pm}(t)\ . $$
They are generated, perhaps artificially,  from  the theta function of  the trivial motive  $\theta_{\Q}$,  using  operations $(\mathbf{T1})-(\mathbf{T4})$.

\subsection{Totally critical values}
Consider  the associated multiple $\Lambda$-functions 
$$\Lambda(\theta_{\pm}, \ldots, \theta_{\pm}; s_1,\ldots, s_r)$$ and their totally critical  values for $s_1=\ldots = s_r=1$.  

From the expressions above we deduce that:
$$ \pi \, \Lambda(\theta_+;1) =   -8   \log(2)  \quad \ , \quad   \Lambda(\theta_{-}; 1) = 0 \ .  $$
The number $\log(2)$ is  a period of a mixed Tate motive over $\Z[\frac{1}{2}]$ (example \ref{exampleramified}). 

\begin{thm} \label{thmMZVasML} Each  of the 
 $2^r$  multiple $\Lambda$ functions  of length $r$
\begin{equation} \label{criticalMZVvalues}  \pi^r  \Lambda(\underbrace{\theta_{\pm}, \ldots, \theta_{\pm}}_r; \underbrace{1,\ldots, 1}_r) 
\end{equation} 
 can be written as a polynomial in $\log(2)$ whose coefficients are multiple zeta values. This polynomial is   homogenous of weight $r$, where $\log(2)$ has  weight $1$ and the weight of a multiple zeta value is the sum of its arguments. Every multiple zeta value  arises in this way.
 Thus,  the  $\Q$-algebra   generated by the numbers \eqref{criticalMZVvalues}  is equal to  the $\Q[\log(2)]$-algebra generated by  multiple zeta values. \end{thm}

\begin{proof} The projective line minus 3 points admits a modular parametrization
$$z:  X_0(4) \overset{\sim}{\To} \Pro^1\backslash \{0,1,\infty\}\ ,$$
where $X_0(4)$ is the quotient of the upper half plane by $\Gamma_0(4)$ and 
\begin{equation} \label{qexpansionofz} z= \left( \frac{\theta_2(\tau)}{\theta_3(\tau)} \right)^4 =16\, q  -128\, q^2 +704\,q^3 -3072\, q^4+\ldots 
\end{equation}
using notation \eqref{thetanull}.
Consider the one forms $\omega_0 = \frac{dz}{z}$ and $\omega_1 = \frac{dz}{1-z}$ on $\Pro^1\backslash \{0,1,\infty\}$.
 By computing finitely many Fourier coefficients, one finds  that 
  $$ \omega_{-} :=  \omega_0 + \omega_1 =   2 \pi i\, \theta_{-}(2 \tau) \, d\tau    $$
  $$ \omega_{+}:= \omega_0  -  \omega_1 =   2 \pi i\, \theta_{+}(2 \tau) \, d\tau    $$
where  $\theta_{\pm}(2\tau)$ are the modular forms of weight $2$ for $\Gamma_0(4)$ given by 
\begin{multline} \label{thetaplusminus} 
\qquad \qquad \theta_+(2\tau)  \  =    \  \theta^4_3 \ = \  8 \, \G_2 (q) - 32\, \G_2(q^4) \\
=1+ 8 \,q +  24 \, q^2+ 32\, q^3 +24 \, q^4 + 48 \, q^5 + \ldots  \qquad 
 \end{multline} 
 \begin{multline} \nonumber 
\qquad \qquad\theta_- (2\tau) \  = \  2 \, \theta_4^4 -  \theta_3^4    \ = \   -24 \, \G_2(q)+96\, \G_2(q^2)- 96\,  \G_2(q^4)  \\
=  1 - 24 \,q +  24 \, q^2 -96\, q^3 +24 \, q^4  -144\, q^5+  \ldots  \qquad  
 \end{multline} 
 Note that $\omega_{-}$ (respectively $\omega_+$) is invariant (resp. anti-invariant) under the involution $z\mapsto 1-z$, which corresponds to the reflection formula $t\mapsto t^{-1}$ in the Mellin variable $ t= \mathrm{Im} ( 2 \tau )$. 
 It follows from the $q$-expansion  \eqref{qexpansionofz} that the unit tangent vector at  the cusp $i \infty$
corresponds to the tangent vector   
 $$16 \frac{\partial}{\partial z}  =  \frac{\partial}{\partial q}\ ,$$
of length $16$ on $\Pro^1\backslash \{0,1,\infty\}$ at the origin.  Similarly,  the image of $\frac{\partial}{\partial q}$ under $t\mapsto t^{-1}$ corresponds to the tangent vector of length $-16$ at $z=1$. It follows from the definition of a multiple $\ML$-value of length $\ell$ as an iterated integral and functoriality that 
\begin{equation}  \label{inprooflambdatheta}
(-\pi)^{\ell} \, \Lambda(\theta_{\pm}, \ldots, \theta_{\pm};1,\ldots, 1) =   \int^{-\overset{\To}{16}_1}_{\overset{\To}{16}_0} \omega_{\pm} \ldots \omega_{\pm} 
\end{equation}  where $\overset{\To}{16}_0$ is the tangent vector of length $16$ at $0$, etc, and the path of integration is 
$$\mathrm{dch}_{16} =    \gamma^{16}_1  \circ \mathrm{dch}\circ   \gamma_0^{16}$$ 
where $\gamma^{16}_0$ is a path inside the tangent space  $T_0 \Pro^1$ from $16$ to $1$, and $\gamma^{16}_1$ is a path inside $T_1 \Pro^1$ from $1$ to $16$. 
By  composition of paths, we deduce that \eqref{inprooflambdatheta} is 
\begin{equation}  \label{composition16}
\sum_{0\leq i \leq j \leq \ell}  \int_{\gamma^{16}_1} \underbrace{ \omega_{\pm} \ldots \omega_{\pm}}_{i}  \times \int_{\mathrm{dch}}  \omega_{\pm} \ldots \omega_{\pm} \times  \int_{\gamma_1^{16}}  \underbrace{\omega_{\pm} \ldots  \omega_{\pm}}_{\ell-j} \ . 
\end{equation}
Since the residues of $\omega_{+}, \omega_{-}$ at zero both equal $1$, the leftmost  integrals reduce to 
$$ \int_{\gamma^{16}_1} \underbrace{ \omega_{\pm} \ldots \omega_{\pm}}_{i}  = \int_{16}^1  \omega_0 \ldots \omega_0 = \frac{ \left(- \log(16)\right)^i}{i!} \ .$$
These integrals take place on the punctured tangent space $\G_m$ of $\Pro^1$ at $0$.  The right-most integrals  of \eqref{composition16} likewise give powers of $\log(2)$, but this time with a sign depending on the integrand (equal to the number of $\omega_{+}$'s). We conclude that  
$$  \pi^{\ell} \, \Lambda(\theta_{\pm}, \ldots, \theta_{\pm};1,\ldots, 1)=  \sum_{0 \leq i \leq j \leq \ell}   \pm 
\frac{\log^{i}(16)}{i!} \frac{\log^{\ell-j}(16)}{(\ell-j)!} \int_{\mathrm{dch}}  \omega_{\pm} \ldots \omega_{\pm}\ , $$
where every  middle integral in  \eqref{composition16} is a period of $\pi^{\mathfrak{m}}_1(\Pro^1\backslash \{0,1,\infty\}, \tone_0, -\tone_1)$ with respect to $\mathrm{dch}$,  which are  multiple zeta values.  Furthermore,   to leading order, 
$$(- \pi)^{\ell} \,  \Lambda(\theta_{\pm}, \ldots, \theta_{\pm};1,\ldots, 1) =  \int_{\mathrm{dch}}  \omega_{\pm} \ldots \omega_{\pm} + \log (16) \Big( \cdots\Big)
$$
where the term  in brackets only involves multiple zeta values of lower weight and powers of $\log(16)$. Since  $\omega_{\pm}$ generate a basis for $H^1_{dR}( \Pro^1\backslash \{0,1,\infty\};\Q)$,  and $2 \log(16) = \pi \Lambda(\theta_{+};1)$, we can  use the shuffle product formula for multiple $\ML$-values and induction  to conclude that every multiple zeta value of weight $w$ can be written  as a $\Q$-linear combination of multiple $\ML$-values 
 $\ML(\theta_{\pm}, \ldots, \theta_{\pm};1,\ldots, 1)$ of length $w$. 
  \end{proof}
 
 For example,
 \begin{eqnarray}
 \pi^2 \,  \Lambda(\theta_{-}, \theta_{+};1,1) & = &   2 \,\zeta(2)- \frac{2}{2!} \log(16)^2  \nonumber \\
 \pi^3 \,  \Lambda(\theta_{-}, \theta_{-}, \theta_{+};1,1,1) & = &  4 \, \zeta(3)  + 2  \log(16) \,  \zeta(2)   -\frac{2}{3!}  \log(16)^3    \nonumber 
 \end{eqnarray} 
 
\begin{rem}
 We showed that   \eqref{criticalMZVvalues} generate all  periods of 
$\pi_1^{\mm}(\Pro^1\backslash \{0,1,\infty\}, \tone_0, -\tone_1)$, which in turn 
 generates all  mixed Tate motives over $\Z$,  by the motivic version of the Deligne-Ihara conjecture \cite{MTZ, BrIhara}.  Therefore all periods of mixed Tate motives over $\Z$ can be expressed as totally critical values of multiple $\ML$-functions. The point is to interpret certain linear combinations of $\frac{dz}{z}$ and $\frac{dz}{1-z}$ as  inverse Mellin transforms of  $L$-functions.
\end{rem}

\subsection{Alternative approaches} One can also   construct all periods of $\MT(\Z)$  as iterated Eisenstein integrals, via Saad's theorem \cite{BrIhara}.
Here is a slightly different  approach. 
Suppose we wish to construct the periods of a bi-extension of $\Q$, $\Q(-3)$, $\Q(-12)$.   By  \cite{BrIhara} \S3, this problem is fiendish when expressed using multiple zeta values. However, the following ad-hoc argument seems to work. The $\ML$-functions of its weight graded pieces are $\xi(s)$, $\xi(s-3)$, $\xi(s-12)$.   By convoluting $\xi(s)$ and $\xi(s-3)$ (respectively $\xi(s-3)$, $\xi(s-12)$), one can obtain $\Lambda(\G_4;s)$ and $\Lambda(\G_{10};s-3)$  as in  example \ref{ExampleEis}. Their critical values capture the periods of extensions of $\Q(-3)$ by $\Q$,  and $\Q(-12)$ by $\Q(-3)$. From these, we obtain  the function $\Lambda(\G_4, \G_{10};s_1,s_2)$.  The combination of totally critical values $\Lambda(\G_4,\G_{10};1,1) - \frac{2520}{691} \Lambda(\G_4,\G_{10};3,5)$  supplies the remaining period (\cite{BrIhara},  \S7.2).

\section{Multiple Riemann $\xi$-function} \label{sectMultipleXi}
The simplest possible example of an iterated Mellin transform is  where all $\theta_i= \theta_{\Q}$ are associated to the trivial motive $\Q= H^0(\mathrm{Spec}\, \Q)$. 

\begin{defn} Define
the multiple Riemann $\xi$-function  to be 
$$\xi(s_1,\ldots, s_r)= \Lambda(\theta_{\Q}, \ldots, \theta_{\Q};s_1,\ldots, s_r) \ .$$
It reduces to the classical Riemann $\xi$-function when $r=1$.
\end{defn}

 Theorem \ref{thmRiemannxi} follows easily from 
theorem \ref{thmLambdaMainProperties},
 together with some simple calculations of residues along the lines of example \ref{exDoubleXi}. 
 
 \begin{rem} The functions $\xi(s_1,\ldots, s_r)$ should not be confused with the multiple zeta functions which are defined for large $\mathrm{Re}(s_i)$ by: 
$$\zeta(s_1,\ldots, s_r) = \sum_{1\leq k_1< \ldots <k_r} \frac{1}{k_1^{s_1} \ldots k_r^{s_r}}\ ,$$
and were essentially first defined by Euler. The proof of their meromorphic continuation to $\C^r$  is much more recent \cite{Zhao}.   They have poles along infinitely many hyperplanes, but a  functional equation valid for all $s_i$ is not  known to my knowledge.  The functions $\xi(s_1,\ldots, s_r)$ and $\zeta(s_1,\ldots, s_r)$ are not related  in any obvious way when $r>1$.
\end{rem}

The positive critical values of the Riemann zeta function are  even integers. Therefore the totally critical positive values of $\xi(s_1,\ldots, s_r)$ 
are also the even integers.

\subsection{Totally even values and multiple quadratic sums}  We can express the totally even positive values of $\xi(s_1,\ldots, s_r)$ in terms 
of the following quantities.

 \begin{defn} For any integers $k_1,\ldots, k_r\geq 1$ define the \emph{multiple quadratic sum}: 
 $$Q(k_1,\ldots, k_r) = \sum_{n_1,\ldots, n_r\geq 1}  \frac{1}{  (n_1^2+\ldots + n_r^2)^{k_1}  \ldots  (n_{r-1}^2+n_{r}^2)^{k_{r-1}} (n_r^2)^{k_r}  }\ .$$
 It converges.  Let us call $2k_1+\ldots +2k_r$ the weight, and  $r$ the depth. 
 \end{defn} 
 
 If $k_1\geq 2$, and  we replace every   exponent  $2$ with a $1$, we obtain
 $$ \sum_{n_1,\ldots, n_r\geq 1}  \frac{1}{  (n_1+\ldots + n_r)^{k_1}  \ldots (n_{r-1}+n_{r})^{k_{r-1}}  n_r^{k_r} }= \sum_{ m_1>\ldots > m_r \geq 1} \frac{1}{m_1^{k_1} m_2^{k_2} \ldots m_r^{k_r}}$$
 which is nothing other than a multiple zeta value  $\zeta(k_r, k_{r-1},\ldots, k_1)$.  

\begin{thm} Let $\ell_i$ be integers $\geq 1$, and set $\ell = \ell_1+\ldots+  \ell_r$. Then 
every totally even  multiple $\xi$-value $ \pi^{\ell} \xi(2\ell_1,\ldots, 2\ell_r)$  is  a $\Q$-linear combination of  multiple quadratic sums $Q(k_1,\ldots, k_p)$ of weight $2\ell$ and depth $\leq r$, i.e., $k_1+ \ldots+k_p= \ell$, and $p \leq r.$
\end{thm}
\begin{proof} Assume that all  $\mathrm{Re}(s_i)>1$, and   define 
\begin{equation}\label{Ds1srdefn} 
D(s_1,\ldots, s_r) = \int_{0 \leq t_1 \leq \ldots  \leq t_r \leq \infty} \theta^0_{\Q}(t_1) t_1^{s_1-1} dt_1 \ldots  \theta^0_{\Q}(t_r) t_r^{s_r-1} dt_r \ .
\end{equation} 
Expanding out the theta functions and exchanging summation and integration gives 
 $$  2^r \sum_{m_1,\ldots, m_r\geq 1} \int_{0\leq t_1 \leq \ldots \leq t_r \leq \infty}  e^{ -\pi (m_1^2t_1^2+\ldots + m_r^2t_r^2)} t_1^{s_1-1} \ldots t_r^{s_r-1} dt_1 \ldots dt_{r}\ . $$
Now write  $s_i= 2\ell_i$. After changing variables $u_i = t_i^2$ this reduces to
 $$  \sum_{m_1,\ldots, m_r\geq 1} \int_{0\leq u_1 \leq \ldots \leq u_r \leq \infty}  e^{ -\pi (m_1^2u_1+\ldots + m_r^2u_r)} u_1^{\ell_1-1} \ldots u_r^{\ell_r-1} du_1 \ldots du_{r}\ . $$
 For any integer $\ell\geq 1$ one has the following identity
 $$\int_{v}^{\infty} e^{-\pi m^2 u} u^{\ell-1} du = \frac{P_{\ell} ( \pi m^2 v)}{ m^{2\ell} \pi^{\ell}} e^{-\pi m^2 v}$$
 where $P_{\ell}$ is a polynomial with integer coefficients of degree $\ell-1$.  If one applies it to the previous integral and integrates out the variables $u_{r}, u_{r-1}, \ldots, u_1$ in turn,  one obtains  an integer linear combination of terms of the form
 $$  \frac{  \pi^{- \ell }}   { ( m_r^2)^{a_r}  ( m_r^2+m_{r-1}^2)^{a_{r-1}} \ldots( m_r^2+m_{r-1}^2+\ldots + m_1^2)^{a_1}  }
 $$
 where $\ell= \ell_1+\ldots + \ell_r = a_1 + \ldots + a_r $ 
 and $1\leq a_r \leq \ell_r$, $1\leq a_{r-1} \leq \ell_r+\ell_{r-1}-1$, 
 $1\leq a_{r-2} \leq \ell_r+\ell_{r-1}+\ell_{r-2}-2$,  and so on. 
 
 This shows that
 $ \pi^{r} D(2 \ell_1,\ldots, 2 \ell_r)$ is an integer  linear combination of 
 $Q(k_1,\ldots, k_r)$  where $k_1+ \ldots+k_r= \ell$.
 To conclude, apply  the definition of the regularised iterated integral with respect to a tangential base point  \S\ref{sectWithPoles} to express $\xi(2\ell_1,\ldots, 2 \ell_r)$ as an isobaric rational linear combination of $ D(2n_1,\ldots, 2n_p)$
 for $p\leq r$.
 \end{proof} 

\begin{example} 

Following the procedure in the previous proof, we find that
$$ D(2\ell_1,2\ell_2)   =   \frac{(\ell_1-1)! (\ell_2-1)! }{ \pi^{\ell_1+\ell_2}} \,  \sum_{k=0}^{\ell_2-1} \binom{\ell_1+k-1}{k} Q(\ell_1+k, \ell_2-k) $$
for $\ell_1,\ell_2 \geq 1$ integers.  In particular, we have:
\begin{eqnarray} \pi^{2} \, D(2,2) \quad   =  &Q(1,1)  \quad =& \sum_{m,n\geq 1}  \frac{1}{(m^2+n^2) n^2}   \quad = \quad \frac{\pi^4}{72}\nonumber \\
 \pi^3\, D(2,4) \quad = &Q(1,2) +Q(2,1)  \quad =  &  \sum_{m,n\geq 1}  \frac{1}{ (m^2+n^2)n^4} +  \frac{1}{ (m^2+n^2)^2 n^2} \nonumber \\ 
 \pi^3\, D(4,2) \quad  = &Q( 2,1)  \quad = &  \sum_{m,n\geq 1}  \frac{1}{ (m^2+n^2)^2 n^2} \nonumber 
 \end{eqnarray} 
 Formula \eqref{lambdain2vars} implies that 
 $$\xi(s_1, s_2 ) = D(s_1, s_2) + \left( \frac{1}{s_1} - \frac{1}{s_2} \right) \xi(s_1+s_2)\ ,$$
which enables us to deduce a formula for $\xi(2\ell_1,2 \ell_2)$. 
For instance, we find that
$$   \frac{\pi^{\ell+1}}{(\ell-1)!} \,  \xi(2\ell , 2) \ =  \left(\sum_{m,n\geq 1}  \frac{1}{(m^2+n^2)^\ell n^{2} }\right)+ \left( \frac{1 -\ell}{2 }  \right)  \left(\sum_{m\geq 1}  \frac{1}{m^{2\ell+2}}\right) $$
 for all values of $\ell\geq 1$.
Conversely, every multiple quadratic  sum of depth $\leq 2$ can be expressed in terms of totally even single and double Riemann $\xi$-values.
 \end{example}

\subsection{Double $\xi$-function and real analytic Eisenstein series}
The function $\xi(s_1,s_2)$  turns out to be a partial Mellin transform of a real analytic Eisenstein series,  which  is defined for $\mathrm{Re} (s)>1$, and $z$ in the upper half plane  by 
 $$ E(z,s) =\frac{1}{2} \sum_{(m,n) \neq (0,0)}  \frac{ y^s }{ |m+n z|^{2s}}  $$ 
 where $y= \mathrm{Im}( z)$. 
Denote its completed version by 
$$\EE(z,s) =  \Gamma_{\R}(2s)  E(z,s)\  $$
where we recall that $\Gamma_{\R}(2s) =  \pi^{-s} \Gamma(s)$. 
It admits a meromorphic continuation to $\C$ with simple poles at $s=0,1$, and satisfies the functional equation
\begin{equation} \label{FneqnforRealEis}
\EE(z,s)= \EE(z,1-s)\ . 
\end{equation} 
Its asymptotic behaviour as $z \mapsto i \infty$ is given by its zeroth Fourier coefficient
$$\EE^{\infty}(z,s)  =  \xi(2s) y^s + \xi(2s-1) y^{1-s} \ .$$
Set 
$\EE^0(z,s) = \EE(z,s) - \EE^{\infty}(z,s)$. With $\tone_{\infty}$ denoting the unit tangent vector at $i\infty$ on the upper half plane, 
a  mild generalisation of   \S\ref{sectIteratedMellin} gives:
$$  \int_{i}^{\tone_{\infty}}  \EE (z, s)  y^{t} \frac{dz}{z}  =   \int_i^{i\infty}  \EE^0 (z,s) y^{t} \frac{dz}{z} - \int_0^1 \EE^{\infty}(z;s)(y)\,   y^{t} \frac{dy}{y} \ . $$
The first integral on the right-hand side converges for all $t$.

\begin{thm}  We have the regularised Mellin transform formula
\begin{equation} \label{xiformula} \xi(2 s_1, 2 s_2) =  \int_{i}^{\tone_{\infty}}  \EE(z, s_1+s_2)  y^{s_2- s_1} \frac{dz}{z}\ .
\end{equation} 
The right-hand side is by definition the ordinary integral
\begin{equation}\label{xiformula2} \int_i^{i \infty}  \EE^0 (z,s_1+s_2) y^{s_2-s_1} \frac{dz}{z}  \ - \   \frac{\xi(2s_1+2s_2)}{2s_2}   \ - \   \frac{\xi(2s_1+2s_2-1)}{1-2s_1}  \ . 
\end{equation}
The integral on the left  admits an analytic continuation to $\C^2$. 
\end{thm} 

\begin{proof} Let $\mathrm{Re}(s_1), \mathrm{Re}(s_2)>\!\!>0$. 
From equation \eqref{lambdain2vars} we find
$$\xi(2s_1,2s_2) = \int_{0\leq t_1\leq t_2 \leq \infty} \theta_{\Q}(t_1) t_1^{2s_1}  \theta^0_{\Q}(t_2) t_2^{2s_2}\, \frac{dt_1}{t_1} \frac{dt_2}{t_2}   - \frac{1}{2s_2} \int_0^{\infty} \theta^0_{\Q}(t_1) t_1^{2s_1+2s_2} \frac{dt_1}{t_1} \ . $$
The second integral is simply $\xi(2s_1+2s_2)$. Therefore
\begin{equation}  \label{inproofxidouble} \xi(2s_1,2s_2) + \frac{1}{2s_2} \xi(2s_1+2s_2)   =  
  \sum_{m\in \Z, n\in \Z \backslash 0} \int_{0\leq t_1\leq t_2 \leq \infty} e^{-\pi (m^2 t_1^2 + n^2 t_2^2)}  t_1^{2s_1} t_2^{2s_2}   \frac{dt_1}{t_1} \frac{dt_2}{t_2} \ .
 \end{equation}
 Change variables by setting $t_1=\lambda , t_2= y \lambda$.  The integral becomes
 $$I_{m,n}= \int_{1}^{\infty} dy \int_0^{\infty}  e^{-\pi (m^2  + n^2 y^2) \lambda^2}  \lambda^{2s_1+ 2s_2}   \frac{d\lambda}{\lambda}    \, y^{2s_2} \frac{dy}{y}  \ . $$
 Perform the $\lambda$ integration using  the following formula
  $$\int_0^{\infty}  e^{- \pi \phi \lambda^2} \lambda^{2s} \frac{d\lambda}{\lambda} = \frac{ 1  }{2 \,\phi^{s}}\, \Gamma_{\R}(2s) \ , $$ 
 which  holds for any $\phi>0$.   This gives 
$$I_{m,n} = \Gamma_{\R}(2s)   \int_{1}^{\infty}   \frac{1}{2}   \left(\frac{y}{m^2+ n^2 y^2} \right)^{s_1+s_2} y^{s_2-s_1}  \frac{dy}{y}  \ .$$
Writing $y  = \mathrm{Im}(z)$ for $z$ on the imaginary axis,  and invoking 
$$m^2+n^2 y^2 = |m+n z|^2$$
we deduce that the right-hand side of \eqref{inproofxidouble} is
$$\Gamma_{\R}(2s) \int_i^{i \infty}   \left( \frac{1}{2}    \sum_{m\in \Z, n\in \Z \backslash 0} \frac{y^{s_1+s_2}}{|m+nz|^{2s_1+2s_2} } \right) y^{s_2-s_1} \frac{dz}{z}\ .$$
It follows from the definitions that
$$\EE(z,s) -\xi(2s)y^s = \Gamma_{\R}(2s)\,  \left(\frac{1}{2}    \sum_{m\in \Z, n\in \Z \backslash 0} \frac{y^{s}}{|m+nz|^{2s} }\right)\ .$$
The left-hand side equals $\EE^0(z,s) + \xi(2s-1) y^{1-s}$. We conclude that
\begin{multline} \xi(2s_1,2s_2) + \frac{1}{2s_2} \xi(2s_1+2s_2)   =  
  \int_i^{i \infty} \left( \EE^0(z,s_1+s_2) +   \xi(2s_1+2s_2-1) y^{1-s_1-s_2} \right) y^{s_2-s_1} \frac{dz}{z}  \\
  =   - \frac{ 1}{1-2s_1}\xi(2s_1+2s_2-1)  +  \int_i^{i \infty}  \EE^0(z,s_1+s_2)   y^{s_2-s_1} \frac{dz}{z}   \end{multline}
 The two terms on the right of \eqref{xiformula2} account for  the poles of $\xi(2s_1,2s_2)$ which lie on $s_2=0, s_1=1$, $s_1+s_2\in \{0, 2\}$. The  apparent singularity at $2s_1+2s_2=1$ cancels out.  It follows that the integral of $\EE^0$ in  \eqref{xiformula2} has no poles. 
\end{proof} 

\subsection{Functional equation} It is instructive to retrieve the   functional equation and shuffle 
product formula for the double $\xi$-function from the previous theorem.
The functional equation
 \eqref{FneqnforRealEis} for  the real analytic Eisenstein series implies that
 $$  \int_{i}^{\tone_{\infty}}  \EE(z, s_1+s_2)  y^{s_2- s_1} \frac{dz}{z} \  =  \  \int_{i}^{\tone_{\infty}}  \EE(z, 1-s_1-s_2)  y^{s_2- s_1} \frac{dz}{z}$$
 which is equivalent, by 
 \eqref{xiformula}, to
$$\xi(2s_1,2s_2)  \ =  \ \xi(1-2s_2,1-2s_1) \ .$$
Thus the functional equation for the double Riemann $\xi$-function follows formally from the functional equation of the real analytic Eisenstein series.

\subsection{Shuffle product}
Similarly, let us compute, using  \eqref{xiformula}, the expression 
$$\xi(2s_1,2s_2) + \xi(2s_2,2s_1) =  \int_{i}^{\tone_{\infty}}  \EE(z, s_1+s_2)  \left(  y^{s_2- s_1}+y^{s_1- s_2}\right) \frac{dz}{z}$$
for large $\mathrm{Re}(s_1), \mathrm{Re}(s_2)$. Since the Eisenstein series is invariant under the involution $S: z\mapsto -  \frac{1}{z}$, the right-hand side can be unfolded and rewritten in the form: 
$$\int_{S\tone_{\infty}}^{\tone_{\infty}}   \EE(z, s_1+s_2)    y^{s_2- s_1} \frac{dz}{z} \ .$$
If one prefers, one can take the lower bound of integration to be $0$ if one assumes that 
$\mathrm{Re}(s_2)>\mathrm{Re}(s_1)$.   The Eisenstein series is itself a Mellin transform
$$2\,  \EE(z, s) = \int_0^{\infty} \left(\Theta_z( t)-1\right) t^{s} \frac{dt}{t}  = \int_0^{\tone_{\infty}} \Theta_z( t)  t^{s} \frac{dt}{t}$$
for $\mathrm{Re}(s)$ large, 
where 
$$\Theta_z(t) = \sum_{m,n \in \Z} e^{-\pi  \frac{|m+n z|^2}{\mathrm{Im}(z)} t} \ .$$
On the imaginary axis $z=it_1$, this theta function factorises:
$$\Theta_{it_1} (t_2) =  \sum_{m,n \in \Z} e^{-\pi  (m^2 t^{-1}_1+  n^2t_1) t_2} = \theta(t_1t_2) \theta(t_1^{-1} t_2)$$ 
where $\theta(x) = \theta_{\Q}(\sqrt{x})$. 
Substituting into the formula above gives 
$$ \xi(2s_1,2s_2) + \xi(2s_2,2s_1) =  \frac{1}{2} \int_{0}^{\tone_{\infty}} \int_{0}^{\tone_{\infty}} \theta( t_1t_2) \theta ( t_1^{-1}t_2)  \,   t_2^{s_1+s_2} t_1^{s_2-s_1} \,  \frac{dt_1}{t_1}\frac{dt_2}{t_2} \ , $$
Change variables  $u=t_1t_2$, $v=t_1^{-1} t_2$ to obtain 
$$    \xi(2s_1,2s_2) + \xi(2s_2,2s_1)  \ = \ \frac{1}{4} \int_{0}^{\tone_{\infty}} \int_{0}^{\tone_{\infty}} \theta(u) \theta(v)  \,   u^{s_2} v^{s_1} \,  \frac{du}{u}\frac{dv}{v}  \ . $$
The right hand side is a product of integrals (after rescaling $u=t^2$, etc) 
$$\int_0^{\tone_{\infty}} \theta_{\Q}(t) t^{2s_1} \frac{dt}{t}   \int_0^{\tone_{\infty}} \theta_{\Q}(t) t^{2s_2} \frac{dt}{t}    \  =  \  \xi(2s_1) \xi(2s_2) \ .$$
We conclude that the shuffle product formula 
$$\xi(2s_1,2s_2) + \xi(2s_2,2s_1) \ = \ \xi(2s_1) \xi(2s_2) $$
is a consequence of the modular invariance (or rather,   invariance with respect to  $z\mapsto - z^{-1}$) of the real analytic Eisenstein series and the factorisation of its associated theta function. 

\subsection{Totally even double $\xi$-values}
The critical values of the  zeta function were computed by Euler.  The next simplest  multiple $\xi$-values  should be the totally even double $\xi$-values.  

\begin{thm} For any integers $\ell_1,\ell_2\geq 1$,  the double $\xi$-values $\xi( 2\ell_1 , 2\ell_2)$ are linear combinations of  regularised Eichler integrals from the point $\tau = i$ to infinity: 
$$
\pi^{-\ell_1-\ell_2} \,\xi(2\ell_1, 2 \ell_2) =   \int_{i}^{\tone_{\infty}}  P_{\ell_1,\ell_2}(\tau) \,  \G_{2 \ell} (\tau)    d\tau 
$$
where $\ell = \ell_1+\ell_2$ and 
$$P_{\ell_1,\ell_2} (\tau)  \quad  \in \quad  \Q \,  i \tau^{2\ell_1-1}   + \quad  \sum_{k=0}^{ 2\ell -2 }\Q  \, \tau^{2k}\ .$$
 \end{thm}

\begin{proof} We  only give the main steps. Starting from \eqref{xiformula}, one can write the real analytic Eisenstein series $\EE(z,\ell)$ in terms of  the function $\mathcal{E}_{\ell-1 , \ell-1}(z)$ where
$$\mathcal{E}_{a,b}(z) =  \frac{w!}{(2\pi i)^{w+1}} \frac{1}{2} \sum_{(m,n)\neq (0,0)} 
\frac{ i \,\mathrm{Im} (z) } { (mz+n)^{a+1} (m \overline{z}+n)^{b+1}}  $$
and   $w=a+b$ is  even $\geq 2$, and $a,b \geq 0$.   These satisfy the differential equations  
$$2 \frac{\partial}{\partial y} \Big( y^s \mathcal{E}_{a,b}(iy)\Big)= y^{s-1}\Big( (a+1)\, \mathcal{E}_{a+1,b-1}(iy)+  (2s-w) \mathcal{E}_{a,b}(iy) + (b+1)\, \mathcal{E}_{a-1,b+1}(iy) \Big)$$
for $a,b\geq 0$ (see \cite{CNHMF}, \S4), where we set $$\mathcal{E}_{w+1,-1}(iy)=\mathcal{E}_{-1,w+1}(iy) = - \frac{2 \pi y}{w+1} \G_{w+2}(iy)\ .$$ 
  We shall show  that  for every $m$ odd, there exist rational numbers $\lambda_{a,b}, \lambda$ such that
\begin{equation}\label{lambdaclaim} 
\frac{\partial}{\partial y}\left( \sum_{a+b=2\ell-2}  y^{m} \lambda_{a,b} \mathcal{E}_{a,b}(iy )\right)  =  y^{m-1} \mathcal{E}_{\ell-1, \ell-1}(iy)  +  \lambda \pi \, y^{m} \, \G_{2 \ell}(iy)   \ .
\end{equation} 
  For this, it suffices to show that the vector $v=(0,\ldots, 0, 1, 0,\ldots, 0)$, with $w/2$ zeros either side of $1$, is in the image of the following matrix for $s=m$ and $w=2\ell-2$: 
$$M_w= \begin{pmatrix}
2s-w & 1 &     & \\
w & 2s-w & 2  &  \\
& w-1 & 2s-w & 3 &  \\
 &&  \ddots &  \ddots \\
% & & 3 & 2s-w & w-1 \\
& & & 2 & 2s-w & w \\
& & & & 1 & 2s-w 
\end{pmatrix}
$$
It   encodes the action of the operator $2 \partial/\partial y$ on the $\{y^s \mathcal{E}_{a,b}(iy)\}$ modulo $\G_{w+2}$. 
Its determinant is $\det (M_w) = 2^{w+1} \, s(s-1) \ldots (s-w)$, so it is unfortunately singular for $s=m$ and $0\leq m \leq w$.
Nonetheless,   it can be written in the form 
$$M_w = \begin{pmatrix}  R_w  & 0 \\ 0 & 0 \end{pmatrix}+\begin{pmatrix}  0  & 0 \\ 0 & \widetilde{R_w} \end{pmatrix}$$
where $R_w$ is a square matrix with  $w/2+1$  rows and columns and 
 $\widetilde{R_w}$ is the matrix $R_w$ rotated through 180 degrees. For example, when $w=4$ we have
$$R_4 =  \begin{pmatrix}
2s-4 & 1 &       \\
4 & 2s-4 & 2    \\
  & 3 & s-2    \\ 
\end{pmatrix} \quad , \quad  
\widetilde{R_4} =  \begin{pmatrix}
s-2 & 3 &       \\
2 & 2s-4 & 4    \\
  & 1 & 2s-4    \\ 
\end{pmatrix}\  . 
$$
Note that the bottom-right entry of $R_w$ is half the other diagonal entries, since it contributes twice to $M_w$. 
One easily checks that the matrices $R_w$ have determinant 
$$\det(R_w) = - 2^{w/2} s(2-s)(4-s)  \ldots (w-s)$$
which is non-zero if $s=m$ is odd. Therefore, for odd $s=m$, the vector $(0,\ldots, 0, 1)$ is in the image of $R_w$, from which it follows that $v$ is in the image of $M_w$.  This proves the claim \eqref{lambdaclaim}.
Substituting \eqref{lambdaclaim} into the integral  \eqref{xiformula} leads to an expression for $\xi(2\ell_1, 2\ell_2)$ as an Eichler integral with an odd power of $\tau$, together with the values of $\mathcal{E}_{a,b}(i)$. Since  
$$d \,\left( \sum_{r+s=w} \mathcal{E}_{r,s}(\tau)(X- \tau Y)^r (X- \overline{\tau} Y)^s \right) = \mathrm{Re}  \left( 2 \pi i \,  \G_{w+2}(\tau) (X - \tau Y)^{w} d\tau \right)  $$ 
the values $\mathcal{E}_{a,b}(i)$ can all be expressed as regularised Eichler integrals from $i$ to infinity of $\G_{w+2} (\tau) \tau^{k}$ where $0\leq k \leq w$ is an even integer. 
\end{proof}

\begin{rem}In particular,   $\pi^{\ell_1+\ell_2} \xi( 2\ell_1 , 2\ell_2)$ are periods of $\pi_1^{\rel} ( \mathcal{M}_{1,1}; i ,   \tone_{\infty})$ , 
the torsor of paths on the moduli stack of elliptic curves $\mathcal{M}_{1,1}$ \cite{HaGPS, MMV}.   In fact,  the regularised Eichler integrals of $\G_{2\ell }$ from $i$ to $i \infty$  are periods of an extension   
$$0 \To  \left(\mathrm{Sym}^{2\ell-2} H^1(E_i)^{\vee} \right) (1) \oplus \Q( 2\ell-1)  \To \mathcal{E} \To \Q \To 0$$
of mixed Hodge structures, 
where $E_i $ is the CM elliptic curve  $ \C/ \Z \oplus  i \Z$.

Using the modularity of $\G_{2\ell}$ under   inversion $\tau \mapsto -1/\tau$, one has
$$(-1)^k \int_i^{\tone_{\infty}}  \tau^{k} \G_{ 2\ell}(\tau) d\tau+  \int_i^{\tone_{\infty}}  \tau^{2\ell-2- k} \G_{ 2\ell}(\tau) d\tau = \int_0^{\tone_{\infty}}  \tau^k \G_{ 2 \ell} (\tau )  d\tau  $$
where $0 \leq k \leq 2\ell-2$. 
The integrals on the right-hand side are known explicitly (\cite{MMV} \S7, \cite{BrIhara}) and are  critical values of $\Lambda(\G_{2\ell};s)$. 
Since most  of the latter  vanish or are rational, this implies relations between the regularised Eichler integrals from $i$ to $i \infty$. 
Furthermore, the Eichler integrals
$$\int_i^{\tone_{\infty}}  \tau^{k} \G_{ 2\ell}(\tau) d\tau$$
for \emph{even} $0\leq k\leq 2\ell-2$ are values of $L$-functions of Hecke Grossencharacters, via their interpretation as values of the real analytic Eisenstein series $\mathcal{E}_{a,b}$ at the CM point $i$.  
\end{rem} 

\begin{ex}  Following the proof of the theorem yields explicit expressions:
\begin{eqnarray}
 \xi(2,2)  &= &  - 8 \, \pi^2   \int_1^{\tone_{\infty}} y\, \G_{4}(iy) dy   = \frac{\pi^2}{72}   \nonumber \\
 \xi(2,4)  &= &  \frac{4\, \pi^3}{3}  \int_1^{\tone_{\infty}} (1+3\, y^2 -4\, y^3)\, \G_{6}(iy) dy  \nonumber \\
  \xi(4,2)  & = &  -  \frac{4\, \pi^3}{3}  \int_1^{\tone_{\infty}} (1- 4\,y + 3\,  y^2)\, \G_{6}(iy) dy  \nonumber \\
  \xi(6,2)  &=  & \frac{8\, \pi^4}{15}  \int_1^{\tone_{\infty}} (1 -4\,  y+5\,y^2) \,\G_{8}(iy) dy  \nonumber 
\end{eqnarray} 
\end{ex}

\begin{rem} Erik Panzer kindly sent  me an independent evaluation of $Q(2,2)$ as a regularised Eichler integral of $\G_8$ by clever application of the Lipschitz summation formula. \end{rem} 

\section{Further comments} \label{sectMThetaV}

\subsection{Multiple Jacobi theta values} The theta function $\theta_{\Q}$ associated to the trivial motive $\Q$ generates a large space  of functions under the operations $(\mathbf{T}1)-(\mathbf{T}4)$ if we also allow  multiplication (remark \ref{RemMult}).  As we have tried to argue, the associated multiple $L$-values  contain some numbers of potential arithmetic interest. 

In order to make this framework more manageable, it is natural to restrict to the graded algebra  $\Theta_J$ generated only by the Jacobi theta-null functions
\begin{equation} \label{thetanull} 
\theta_2(\tau)  =  \sum_{n\in \Z} q^{(n+\frac{1}{2})^2}  \quad , \quad 
\theta_3(\tau)  =  \sum_{n\in \Z} q^{n^2}  \quad , \quad  \theta_4(\tau)  =   \sum_{n\in \Z} (-1)^n q^{n^2}   , 
\end{equation} 
where $q= e^{i \pi \tau}$. They are not  self-dual, and can  have half-integral weights. Their Mellin transforms can have poles at half-integers, so we should allow $\theta^{\infty} \in \C[\sqrt{t}]$ in \S\ref{sectIteratedMellin} (2). 

Let us call a \emph{multiple Jacobi theta value} a totally critical value of 
$$\ML(\theta_1,\ldots, \theta_r; s_1,\ldots, s_r) \qquad \hbox{ where } \quad  \theta_i \in \Theta_J\  .$$
The structure of these numbers should prove to be  interesting, given the intricate algebraic and differential relations which theta functions  satisfy \cite{ZudilinTheta}.
\begin{enumerate} 
\item  By  \S\ref{sectMTZ} and \eqref{thetaplusminus}, multiple zeta values are examples of  multiple Jacobi theta values.
\vspace{0.1in}
\item Since    $\theta_{\Q}(\sqrt{t}) = \theta_3( i t)$, and $t \mapsto \sqrt{t}$ corresponds to $s\mapsto 2s$, 
 the totally even Riemann $\xi$-values are also multiple Jacobi theta values. 
\vspace{0.1in}
\item Since iterated integrals are invariant under reparametrisation (note that  this can affect tangential base points), and since the Eisenstein series $\G_4(2\tau)$, $\G_6(2\tau)$ are  in $\Theta_J$ (see \cite{ZudilinTheta}) and generate the full  ring of modular forms of level one, we deduce that all  totally holomorphic multiple modular values  for $\mathrm{SL}_2(\Z)$ \cite{BrIhara} are  multiple Jacobi theta values.  
\vspace{0.1in}

\item  Multiple Jacobi values of length one   are   related to  `lattice sums' which arise in a variety of contexts (see \cite{Lattice} for interesting examples and references) and to  values of the Arakelov double zeta function  $Z_{\Q}(w,s)$ of \cite{ArakelovZeta}.
\end{enumerate}

Thus all the numbers discussed in  the first half of this talk \cite{BrIhara} are subsumed  into this class. 
Note that multiple zeta values arise in two completely different ways: via $(1)$, but also via $(2)$ as   iterated integrals of Eisenstein series by Saad's theorem \cite{BrIhara}.

\subsection{$L$-functions of non-holomorphic modular forms}
In \cite{CNHMF} we  defined non-holomorphic modular forms  by taking real and imaginary parts of iterated primitives of classical holomorphic modular forms. 
The prototypical examples are the real-analytic Eisenstein series $\mathcal{E}_{a,b}(s)$ considered in \S\ref{sectMultipleXi}.  These functions, via a regularised Mellin transform, also give rise to $L$-functions with good properties   \cite{CNHMF}, \S9.4.   One can show that they are linear combinations of  multiple $\ML$-functions
$$\ML(f_1,\ldots, f_{r-1}, f_r; p_1,\ldots, p_{r-1}, s)$$
where $f_1,\ldots, f_r$ are modular forms of full level, and $p_1,\ldots, p_{r-1}$ are fixed integers which are critical for each $f_i$,  and only the last parameter is allowed to vary.   This fact was one of our motivations for the present work, but will be discussed elsewhere.

\subsection{Conclusion.} We have defined a family of multiple motivic $\ML$-functions with good properties and exhibited  examples where their totally critical values are related to periods. For motives of higher rank $>2$,  or in the case of  several  motives which have different  gamma factors, there is currently insufficient evidence  to know if the definition needs modifying in some way.  Hilbert modular forms, for example,  have multi-variable Mellin transforms which define $L$-functions in several variables. It is not clear how these should relate to the objects defined here. In any case, the present definition \S\ref{sectIteratedMellin}, applied to motives of low rank,  leads to new objects, such as the multiple Riemann $\xi$-function, which are potentially of interest.

 \vspace{0.1in}
 
 \emph{Acknowledgements}. This paper was written at All Souls college, Oxford and was partially supported by ERC grant  724638. Many thanks to Erik Panzer for discussions and enthusiasm, and M. Kaneko for corrections.
   \\

\noindent  Francis Brown, \\
All Souls College, \\
Oxford, \\
OX1 4AL,\\
United Kingdom.  \\
\texttt{francis.brown@all-souls.ox.ac.uk}

\bibliographystyle{plain}
\bibliography{main}

\end{document}